\documentclass[a4paper,12pt,twoside]{article}

\usepackage[body={175mm,235mm}]{geometry}
\usepackage{a1}
\usepackage[english]{babel}
\usepackage[latin1]{inputenc} 
\usepackage{amsfonts,amssymb}
\usepackage{amsmath}
\usepackage{graphicx}
\usepackage{float}
\usepackage[pdftex]{hyperref}
\usepackage{url}

\usepackage{makeidx}
\makeindex

\def\mapright#1#2#3{\smash{\mathop{\hbox to
#3{\rightarrowfill}}\limits^{#1}_{#2}}}

\def\mapleft#1#2#3{\smash{\mathop{\hbox to
#3{\leftarrowfill}}\limits^{#1}_{#2}}}

\def\mapright#1#2{\smash{\mathop{\hbox to 0.90cm{\rightarrowfill}}\limits^{#1}_{#2}}}
\def\mapleft#1#2{\smash{\mathop{\hbox to 0.90cm{\leftarrowfill}}\limits^{#1}_{#2}}}

\def\mapleftright#1#2{\smash{\mathop{\hbox to 0.80cm{\leftarrowfill \rightarrowfill}}\limits^{#1}_{#2}}}

\title{PL-embedding the dual of two Jordan curves \\ 
into $\mathbb{S}^ 3$ by an $O(n^ 2)$-algorithm
} 
\author{Sóstenes L. Lins and Ricardo N. Machado}

\date{\today}

\begin{document}

\maketitle

\begin{abstract}
Let be given a {\em colored 3-pseudo-triangulation} $\mathcal{H}^\star$ with $n$ tetrahedra.
Colored means that each tetrahedron have vertices 
distinctively colored 0,1,2,3.
In a {\em pseudo} 3-triangulation the intersection of simplices might be 
subsets of simplices
of smaller dimensions (faces), instead of a single maximal face, 
as for true triangulations. If $\mathcal{H}^\star$ is the dual of a cell
3-complex induced (in an specific way to be made clear) 
by a pair of Jordan curves with $2n$ transversal crossings,
then we show that the induced 3-manifold $|\mathcal{H}^\star|$ is 
$\mathbb{S}^3$ and we make available an $O(n^2$)-algorithm to produce a
PL-embedding (\cite{rourke1982introduction}) of $\mathcal{H}^\star$ into $\mathbb{S}^3$.
This bound is rather surprising because such PL-embeddings are often of exponential size.
This work is the first step towards obtaining, via an $O(n^ 2)$-algorithm,
a framed link presentation inducing the same closed orientable 3-manifold
as the one given by a colored pseudo-triangulation. 
Previous work on this topic appear in
\url{http://arxiv.org/abs/1212.0827}, \cite{linsmachadoA2012}, 
\url{http://arxiv.org/abs/1212.0826}, \cite{linsmachadoB2012} and 
\url{http://arxiv.org/abs/1211.1953}, \cite{linsmachadoC2012}.
However, the exposition and the new proofs of this paper are meant to be entirely self-contained.

\end{abstract}

\section{Introduction}
The subject of this paper is motivated by the following problem:
given a triangulated closed, connected, orientable 3-manifold $\mathbb{R}^3$, 
how to obtain by a polynomial algorithm  a {\em blackboard-framed link presentation}
inducing  $\mathbb{R}^3$. This is one of the most basic open problems in
3-manifold topology: there are two main languages to present a specific 3-manifold.
The triangulation based presentations (triangulations, Heegaard diagrams, gems, special
spines, etc) and the framed link presentations (various types of decorated links which
yield the 3-manifold as a recipe for surgery on $\mathbb{S}^3$). Going from the
second type presentation to a first type one is straightforward by a linear algorithm,
see Section 2 of \url{http://arxiv.org/abs/1305.5590}, \cite{lins2013C}.
However there is no known polynomial algorithm to go from the first to the second.
This paper starts to remedy this situation.

A blackboard-framed link is in 1-1 correspondence with a {\em  blink}, which is
simply a plane graph with an arbitrary bipartition of the edges set into black and gray edges.
For a recent combinatorially oriented 
account about how a blink induces a 3-manifold, see 
\url{http://arxiv.org/abs/1305.4540}, \cite{lins2013B}. 
The blink of the bottom part of Fig. \ref{fig:euclideans} provides a solution for a problem
left open in \url{http://arxiv.org/arXiv:math/0702057}, \cite{lins2007blink}.
Even though the general problem remains unconquered, substantial advances were
possible and, in practice, it permits to find by an $O(n^2)$-algorithm a 
blink presentation inducing the same 3-manifold
as a given gem. From here on we concentrate in producing an embedding into $\mathbb{S}^3$ 
of a $J^2$-gem. This result is used as a central important lemma in finding the 
blink equivalent to a gem, which is the subject of 
subsequent papers currently being polished.

\begin{figure}[H]
\begin{center}
\includegraphics[width=10.5cm]{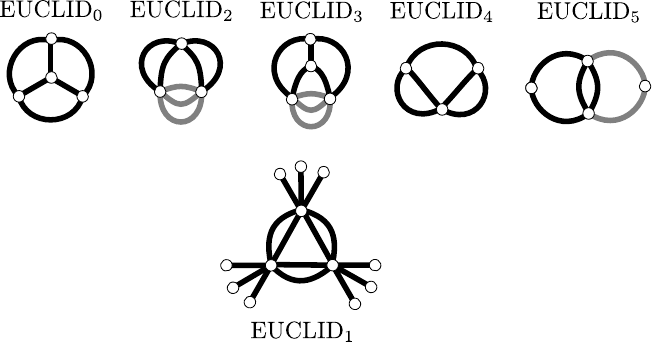} \\
\caption{\sf The discovery of a blink for EUCLID$_1$, solving an
open problem left in page 117 of \cite{lins2007blink}:
of the six euclidean 3-manifolds only EUCLID$_1$ did not have a blink presentation. 
To find this blink it is necessary to develop 
and understanding some deep geometric properties of gems, 
as we start doing in in this paper. The blink presentation for EUCLID$_1$
was obtained in this way.
}
\label{fig:euclideans}
\end{center}
\end{figure}

\subsection{$J^2$-gems}
A {\em $J^ 2$-gem} is a 4-regular, 4-edge-colored planar graph $\mathcal{H}$ obtained from the
intersection pattern of two Jordan curves $X$ and $Y$ with $2n$ transversal
crossings.
These crossings define consecutive segments of $X$ alternatively 
inside $Y$ and outside $Y$. Color the first type 2 and the second type 3.
The crossings also define consecutive segments of $Y$ alternatively 
inside $X$ and outside $X$. Color the first type 0 and the second type 1.
This defines a 4-regular 4-edge-colored graph $\mathcal{H}$ where the vertices are the crossings
and the edges are the colored colored segments, see Fig. \ref{fig:j2gemr245}.

\begin{figure}[H]
\begin{center}
\includegraphics[width=5cm]{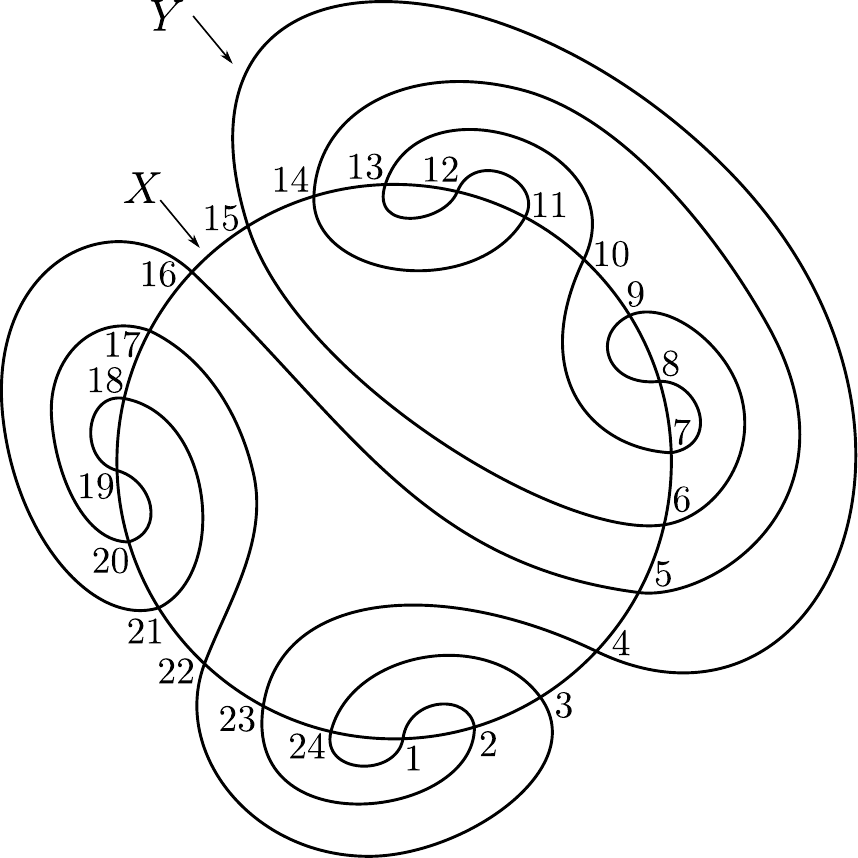} \\
\caption{\sf How to interpret a pattern formed by two $2n$-crossing Jordan curves as a gem:
start labeling the crossings of $X$ in consecutively in a counterclockwise 
way so that edge linking 1 to 2 is of color 2 which alternates with color 3.
Let the segments of $Y$ internal to $X$ be colored 0 and the ones external to $X$
be colored 1. The result is, by definition, a $J^2$-gem.
}
\label{fig:j2gemr245}
\end{center}
\end{figure} 

Let $\mathcal{H}^\star$ be the 3-dimensional abstract 3-complex formed by taking a set of 
vertex colored tetrahedra in 1--1 correspondence with the set of vertices of $\mathcal{H}$,
$V(\mathcal{H})$, so that each tetrahedra
has vertices of colors 0,1,2,3. This vertex coloring induces a face coloring of the 
triangular faces of the tetrahedron: color $i$ the face opposite to the vertex colored $i$.
For each $i$-colored edge of $\mathcal{H}$ with ends $u$ and $v$
paste the corresponding tetrahedra $\nabla_u$ and $\nabla_v$ so as to paste the two triangular
faces that do not contain a vertex of color $i$ in such a way as to
to match vertices of the other three colors. 
We show that the topological space 
$|K|$ induced by $\mathcal{H}^\star$ is $\mathbb{S}^ 3$. Moreover we describe an $O(n^2)$-algorithm
to make available a PL-embedding (\cite{rourke1982introduction}) of $\mathcal{H}^\star$ 
into $\mathbb{S}^3$. We get explicit coordinates in $\mathbb{S}^ 3$ for the 0-simplices and
the p-simplices $(p \in \{1,2,3\}$) become linear simplices in the spherical geometry. For basic
topological definitions we refer to the lucid book by Stillwell \cite{stillwell1993classical}.

\subsection{Gems and their duals}
A {\em (3+1)-graph} \index{(3+1)-graph} $\mathcal{H}$ is a connected regular graph of degree 4 where 
to each vertex there are four incident differently colored edges in the color 
set $\{0,1,2,3\}$.
\index{residue}
For $I \subseteq \{0,1,2,3\}$, an {\em $I$-residue} is a component of the subgraph 
induced by the $I$-colored edges. Denote by $v(\mathcal{H})$ the number of
$0$-residues (vertices) of $\mathcal{H}$. For $0\le i < j\le3$, an $\{i,j\}$-residue is also 
called an {\em $ij$-gon} or an $i$- and $j$-colored {\em bigon} (it is an even polygon, where the edges are 
alternatively colored $i$ and $j$). Denote by $b(\mathcal{H})$ the total number of $ij$-gons 
for $0\le i < j\le3$. Denote by $t(\mathcal{H})$ the total number of 
$\overline{i}$-residues for $0\le i\le 3$, where $\overline{i}$ 
means complement of $\{i\}$ in $\{0,1,2,3\}$.

We briefly recall the definition of gems taken from \cite{lins1995gca}.
A {\em $3$-gem} \index{3-gem} is a $(3+1)$-graph $\mathcal{H}$
satisfying $v(\mathcal{H})+t(\mathcal{H})=b(\mathcal{H})$. This relation is equivalent to having the vertices, 
edges and bigons 
restricted to any $\{i,j,k\}$-residue inducing a plane graph where the 
faces are bounded by the 
bigons. Therefore we can embed each such $\{i,j,k\}$-residue into a sphere 
$\mathbb{S}^2$. We consider
the ball bounded this $\mathbb{S}^2$ as induced by the  $\{i,j,k\}$-residue. For this reason
an $\{i,j,k\}$-residue in a 3-gem, $i<j<k$, is also called a \index{triball} {\em triball}.
An $ij$-gon appears once in the boundary of triball $\{i,j,k\}$ and once in the boundary
of triball $\{i,j,h\}$. By pasting the triballs along disks bounded by all 
the pairs of $ij$-gons, $\{i,j\} \subset \{0,1,2,3\}$ of a gem $\mathcal{H}$,
we obtain a closed 3-manifold denoted by $|\mathcal{H}|$. 
This general construction is dual to the one 
exemplified in the abstract and produces any closed 3-manifold. 
The manifold is orientable if and only if 
$\mathcal{H}$ is bipartite, \cite{lins1985graph}. A {\em crystallization} is a
gem which remains connected after deleting all the edges of any given color, that is,
it has one $\{i,j,k\}$-residue for each trio of colors $\{i,j,k\} \subset \{0,1,2,3\}$.

Let $\mathcal{H}^\star$ be the dual of a gem $\mathcal{H}$.
An $\overline{i}$-residue of $\mathcal{H}$ corresponds in $\mathcal{H}^\star$ to a $0$-simplex of $\mathcal{H}^\star$. Most 0-simplices of $\mathcal{H}^\star$
do not correspond to $\overline{i}$-residues of $\mathcal{H}$.
An $ij$-gon of a gem $\mathcal{H}$ corresponds in $\mathcal{H}^\star$ to a {\em PL1-face} formed by a sequence of
1-simplices of $\mathcal{H}^\star$; this PL1-face is the intersection of two PL2-faces of colors $i$ and $j$; their
two bounding $0$-simplices correspond to an $\overline{h}$- and to a $\overline{k}$-residue, where
$\{h,i,j,k\}=\{0,1,2,3\}$.
An $i$-colored edge of $\mathcal{H}$ corresponds to a {\em PL2-face} 
which is a 2-disk triangulated by a subset of $i$-colored 2-simplices of $\mathcal{H}^\star$. Finally to a vertex of $\mathcal{H}$, 
it corresponds a {\em PL3-face} of $\mathcal{H}^\star$ which is a 3-ball formed by a subset of 3-simplices of $\mathcal{H}^\star$.

\begin{proposition}
\label{prop:j2shere}
The $3$-manifold induced by a $J^2$-gem $\mathcal{H}$ is $\mathbb{S}^3$.
\end{proposition}
\begin{proof}
Removing from $\mathcal{H}$ all the edges of any given color produces a restricted graph
which is connected (whose faces are bounded by the $ij$-gons). 
So $\mathcal{H}$ has four 3-residues, one of each type. 
Denote by $b_{ij}$ the number of $ij$-gons of $\mathcal{H}$.
Each one of these residues are planar graphs
having $v=2n$ vertices, $3v/2$ edges and $b_{12}+b_{13}+b_{23}$, 
$b_{02}+b_{03}+b_{23}$, $b_{01}+b_{13}+b_{03}$ and  $b_{12}+b_{01}+b_{02}$ faces
for, respectively, the $\overline{0}$-, $\overline{1}$-, $\overline{2}$-, 
$\overline{3}$-residue. Adding the four formulas for the Euler characteristic of the sphere
imply that $v(\mathcal{H})+4=b(\mathcal{H})$. Therefore, $\mathcal{H}$ is a crystallization having one $0i$-gon and
one $jk$-gon. This implies that the fundamental group of the induced manifold is trivial:
as proved in \cite{lins1988fundamental}, the fundamental group of the space induced by a crystallization
is generated by $b_{0i}-1$ generators, and in our case this number is 0.
Since Poincaré Conjecture is now proved, we are done. 
However, we can avoid using this fact and, as a bonus, obtaining the validity
of the next corollary, which is used in the sequel.

Assume that $\mathcal{H}$ is a $J^2$-gem which does not induce $\mathbb{S}^3$ and has the smallest possible
number of vertices satisfying these assumptions. By planarity we must have a pair of edges of
$\mathcal{H}$ having the same ends $\{p,q\}$. Consider the graph $\mathcal{H} fus \{p,q\}$ obtained from
$\mathcal{H}$ by removing the vertices $p$, $q$ and the 2 edges linking them as well as welding the 2
pairs of pendant edges along edges of the same color. In \cite{lins1985simple} S. Lins proves that
if $\mathcal{H}$ is a gem, $\mathcal{H}'=\mathcal{H} fus \{p,q\}$ is also a gem and that two exclusive 
relations hold regarding $|\mathcal{H}|$ and $ |\mathcal{H}'|$, their induced 3-manifolds:
either $|\mathcal{H}| = |\mathcal{H}'|$ in the case that $\{p,q\}$ induces a 2-dipole
or else  $|\mathcal{H}| = |\mathcal{H}'| \# (\mathbb{S}^2\times \mathbb{S}^1)$. Since
$\mathcal{H}'$ is a $J^2$-gem, by our minimality hypothesis on $\mathcal{H}$
the valid alternative is the second. But this
is a contradiction: the fundamental group of $|\mathcal{H}|$ would not be trivial, 
because of the summand $\mathbb{S}^2\times \mathbb{S}^1$. 
\end{proof}

\subsection{Dipoles, pillows and balloons} 

Suppose there are $m$ edges linking vertices $x$ and $y$ of a gem, $m \in \{1,2,3\}$.
We say that $\{u,v\}$ is an {\em $m$-dipole} if removing all edges in the colors of the ones linking $x$ to $y$,
these vertices are in distinct components of the graph induced by the edges in the complementary
set of colors. To {\em cancel the dipole} means deleting the subgraph induced by $\{u,v\}$ and identify pairs 
of the hanging edges along the same remaining color. To {\em create the dipole} is the inverse operation.
It is simple to prove that the manifold of a gem is invariant under dipole cancellation or creation.
Even though is not relevant for the present work state the foundational result on gems: two $3$-manifolds
are homeomorphic if and only if any two gems inducing it are linked by a finite number of
cancellations and creations of dipoles, \cite{ferri1982crystallisation, lins2006blobs}.

The dual of a $2$-dipole $\{u,v\}$, with internal colors $i,j$ is named a {\em pillow}. 
It consists of two PL3-faces $\nabla_u$ and $\nabla_v$ sharing two PL2-faces colored $i$ and $j$.
The {\em thickening of a 2-dipole into a 3-dipole} is defined as follows. Let
$i,j$ be the two colors internal to 2-dipole $\{u,v\}$ and $k$ a third color.
Let $a$ be the $k$-neighbor of $x$ and $b$ be the $k$-neighbor of $y$. Remove
edges $[a,x]$ and $[b,y]$ and put back $k$-edges $[u,v]$ and $[r,s]$. This completes
the thickening. It is simple to prove that the thickening in a gem produces a gem. 
We must be careful because the inverse blind inverse operation {\em thinning a 3-dipole} 
not always produces a gem. The catch is that the result of the thinning perhaps is not a 2-dipole.
In this sense the thinning move is not local: we must make sure that the result is a 2-dipole.
To the data needed in thinning a 3-dipole $\{u,v\}$ with internal colors $\{i,j,k\}$ 
we must add the $k$-edge $[r,s]$. Note that the $k$-edges $[u,v]$ and $[r,s]$ are in the same $hk$-gon,
where $h$ is the fourth color. Denote by $\Delta_{rs}$ the dual of $[r,s]$.
Let $\nabla_u \cup \nabla v \cup \Delta_{rs}$ be called a {\em balloon}. 
Note that it consists of 2 PL3-faces $\nabla_u$ and $\nabla_v$ sharing three PL2-faces 
in colors $\{i,j,k\}$ together
with a $k$-colored PL2-face whose intersection with  $\nabla_u \cup \nabla_v$ is 
a PL1-face corresponding to the dual of the $hk$-gon, where $h$ is the fourth color.
Let $\nabla_u \cup \nabla_v$ be the {\em balloon's head} and let $\Delta_{rs}$
be the {\em balloon's tail}.

\subsection{The Strategy for finding the $O(n^2)$-algorithm}
We want to find a PL-embedding for the dual $\mathcal{H}^\star$ of a $J^2$-gem $\mathcal{H}$ 
into $\mathbb{S}^3$. To this end we remove one PL3-face of $\mathcal{H}^\star$ 
(one vertex of $\mathcal{H}$) and find
a PL-embedding in $\mathbb{R}^3$  forming a a PL-triangulated tetrahedron. 
After we use the inverse of a stereographic projection with center
in the exterior of the triangulated tetrahedron. In this way we recover in $\mathbb{S}^3$
the missing PL3-face.

In this work we describe the PL-embedded PL3-faces of $\mathcal{H}^\star$ 
into $\mathbb{R}^3$ by making it geometrically clear that its boundary
is a set of 4 PL2-faces, one of each color, forming an embedded $\mathbb{S}^2$ whose interior is
disjoint from the interior of 
$\mathbb{S}^2$'s corresponding to others PL3-faces. Thus, for our purposes
it will be only necessary to embed the 2-skeleton of $\mathcal{H}^\star$.

A direct approach to find the PL-embedding of the dual of a general $J^2$-gem with $2n$ vertices,
seems very hard. Therefore, we split the algorithm into 4 phases. 

In the first phase we find a sequence of $n-1$ 2-dipole thickenings into
3-dipoles not using color 3, where the new involved edge is either 0 or 1 so that the 
final gem is simply a circular arrangement 
of $n$ 3-dipoles with internal
colors $0,1,2$. Such a canonical $n$-gem is 
named a {\em bloboid} and is denoted $\mathcal{B}_n$. 
Such 3-dipoles are also named a {\em blob over a 3-colored edge}.
A $J^2B$-gem is a gem that, after canceling blobs over $3$-colored edge becomes
a $J^2$-gem. 
This indexing decreasing sequence is easily obtainable from the
primal objects, in the case, simplifying $J^2B$-gems until the bloboid is obtained: 
$$\mathcal{H}=\mathcal{H}_{n} \mapright{2dip}{thick_1} \mathcal{H}_{n-1} \mapright{2dip}{thick_2} \ldots 
\mapright{2dip}{thick_{n-1}} \mathcal{H}_{1}=\mathcal{B}_n=\mathcal{B}.$$ 

In the second phase we first find (phase 2A) specific abstract PL-triangulations, 
for the PL2-faces for the index increasing sequence of abstract colored 2-dimensional PL-complexes.
Each of these complexes, latter, are going to be PL-embedded into $\mathbb{R}^3$ so 
that the PL2-faces are topologically 2-spheres with disjoint interior. Attaching 
3-balls bounded by these spheres we get the dual of the $J^2$-gem $\mathcal{H}$ (with a vertex
removed).
$$\mathcal{B}^\star=\mathcal{H}^\star_{1} \mapright{bp}{move_1}  \mathcal{H}^\star_{2}  
\mapright{bp}{move_{2}} \ldots
\mapright{bp}{move_{n-1}} \mathcal{H}^\star_{n}=\mathcal{H}^\star.$$ 
In parallel to the construction of the sequence $\mathcal{H}^\star_{m}$'s we also construct
(phase 2B) a sequence of {\em wings} $\mathcal{W}_{m}$'s and their {\em nervures} 
$\mathcal{N}_{m}$'s so that  
$\mathcal{S}_m=\mathcal{W}_{m} \cup \mathcal{N}_{m}$,
called {\em a strut}, is an adequate planar graph, defined recursively. 
Moreover, wings and nervures are partitioned into their {\em left} and {\em right} parts: 
$\mathcal{W}_m=\mathcal{W}'_m \cup \mathcal{W}''_m$ and
$\mathcal{N}_m=\mathcal{N}'_m \cup \mathcal{N}''_m$. The sequence of struts is
$$\mathcal{S}_1^\star=\mathcal{W}_{1} \cup \mathcal{N}_{1} 
\mapright{wbp}{move_{1}} 
\mathcal{S}_2^\star=\mathcal{W}_{2} \cup \mathcal{N}_{2} \mapright{wbp}{move_{2}} \ldots
\mapright{wbp}{move_{n-1}} 
\mathcal{S}_n^\star=\mathcal{W}_{n} \cup \mathcal{N}_{n}=
\mathcal{W} \cup \mathcal{N}=\mathcal{S}^\star.$$ 
Each wing $\mathcal{W}_m$'s corresponds to a section 
of the previous sequence $\mathcal{H}_m^\star$'s by two adequate 
fixed semi-planes $\Pi'$ and $\Pi''$.
The construction of the struts (which are planar graphs) is
recursive. Initially $\mathcal{W}_1^\star$ is a set of $2n$ lines incident to $a_1$ 
and a set of $2n$ lines incident to $b_1$, while $\mathcal{N}_1^\star$ 
is $\varnothing$.  Going from $\mathcal{S}^\star_{m}$
to $\mathcal{S}_{m+1}^\star$ is very simple: two new vertices and four
new edges appear, so as to maintain planarity.

In the third phase we make the abstract final element 
$\mathcal{W}^\star \cup \mathcal{N}^\star$
of the second phase {\em rectilinearly} (that is each edge is a straight line segment) PL-embedded. 
By a cone construction we obtain from the rectilinearly PL-embedded strut
$\mathcal{S}^\star$ a special
PL-complex, named $\mathcal{H}_1^\diamond$. This complex does not correspond to a gem dual
and it can be loosely explained as $\mathcal{H}_1^\star$ with all balloon's heads ``opened''.

The fourth phase, the {\em pillow filling phase} 
starts with $\mathcal{H}_1^\diamond$ and the uses the abstract sequence 
$\mathcal{H}^\star_{m}$'s to produce 
a pillow filling sequence $$\mathcal{H}^\diamond_{1}  \mapright{pillow}{filling_{1}}  
\mathcal{H}^\diamond_2  \mapright{pillow}{filling_{2}} \ldots  \mapright{pillow}{filling_{n-1}}
\mathcal{H}^\diamond_{n}=\mathcal{H}^\star.$$
In this phase everything is embedded into $\mathbb{R}^3$ and the last 
element, $\mathcal{H}^\star$,
is the PL-embedding that we seek: the PL-embedding of the dual of the original
$J^2$-gem $\mathcal{H}$ minus a vertex into $\mathbb{R}^3$.

The whole procedure can be implemented as a formal 
algorithm that takes $O(n^2)$-space
and $O(n^2)$-time complexity, where $2n$ is the number of vertices of the original $J^2$-gem.

\subsection{A complete example}
In the example corresponding to Fig. \ref{fig:winglist01}, \ldots, Fig. \ref{fig:winglist06}
we gather and display the data structure we need for the $O(n^2)$-algorithm
to PL-embed the dual of the $J^2$-gem $\mathcal{H}_{12}$ of Fig. \ref{fig:winglist06}. 
The wings, nervures and their union, the struts, are
partitioned into left and right parts:
$$\mathcal{W}_m =  \mathcal{W}'_m \cup  \mathcal{W}''_m, \hspace{4mm} 
\mathcal{N}_m =  \mathcal{N}'_m \cup  \mathcal{N}''_m, \hspace{4mm}
\mathcal{S}^\star_m =  \mathcal{S}^{\star'}_m \cup  \mathcal{N}''_m, \hspace{4mm} 
$$

The sequence of figures displays the sequences of three data structures that we need:
the inverse sequence of
$J^2B$-gems $\mathcal{H}_m$, the colored 2-dimensional 
complexes $\mathcal{H}^\star_m$ 
(partitioned into colored PL2-faces)
and the struts $\mathcal{S}_m = \mathcal{W}_m \cup \mathcal{N}_m$. 
The initial data structures are very simple and, in the figures we make clear
how to obtain the next term of these sequences. In particular, the next colored 
2-complex is given as the result of a $bp$-move, which changes a {\em balloon} into a
{\em pillow}. See definitions in Subsection \ref{subsec:primaldualmoves}.

\begin{figure}[H]
\begin{center}
\includegraphics[width=14.7cm]{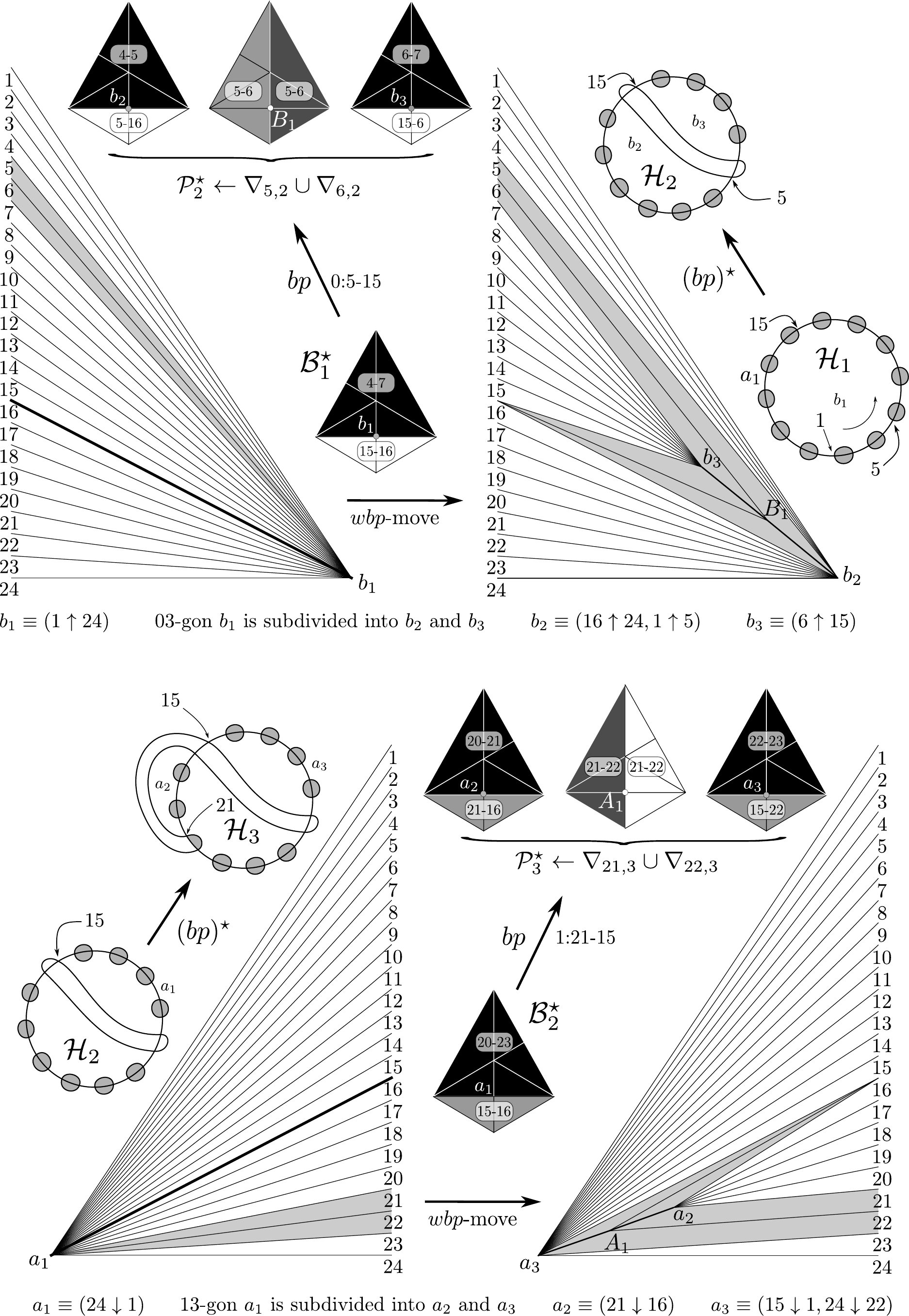} \\
\caption{\sf The initial right wing $\mathcal{W}''_1$
is a set of lines emanating from $b_1$.
The initial right nervure $\mathcal{N}''_1$ is empty. 
The initial left wing $\mathcal{W}'_1$
is a set of lines emanating from $a_1$.
The initial left nervure $\mathcal{N}'_1$ is empty.
At the end of each $wpb$-move a pair of edges is added to the nervure.
Lower case symbols $a_j, b_k$ refer to 13-gons and 03-gons. Upper case symbols
$A_j, B_k$ are auxiliary 0-simplexes.
}
\label{fig:winglist01}
\end{center}
\end{figure}

\begin{figure}[H]
\begin{center}
\includegraphics[width=14.1cm]{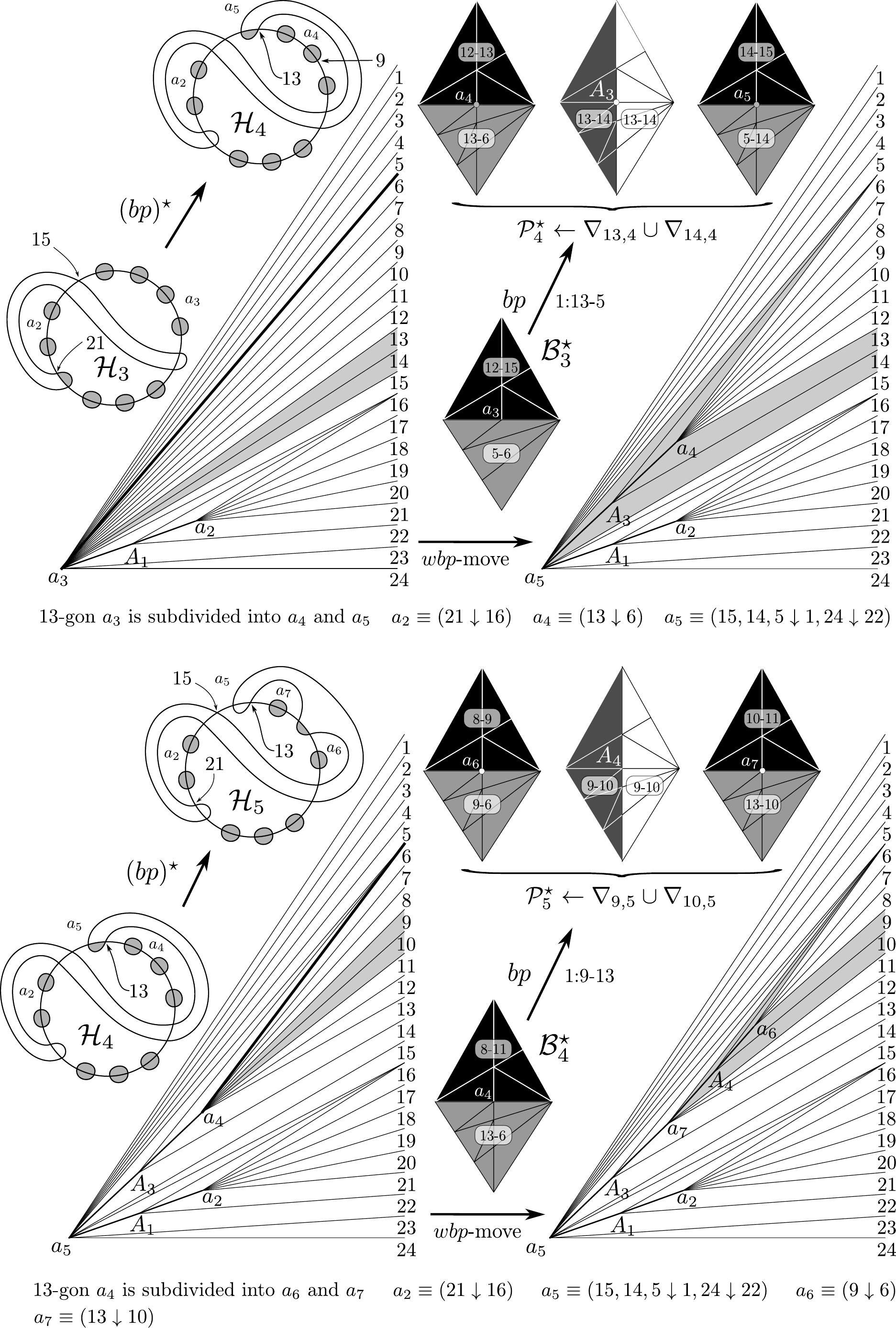} \\
\caption{\sf A left strut is modified by a $wbp$-move. What is needed as input 
is a pair of adjacent 
shaded triangles and a thick edge. In the lower the left strut is further 
modified by another $wbp$-move. The modification of balloons into pillows define the new
colored abstract combinatorial complexes. The nervures $\mathcal{N}_m$ are auxiliary devices
that will be disposed after we find the rectilinear embedded
$\mathcal{W}$ by a deterministic linear algorithm,
See Fig. \ref{fig:winglist06}.}
\label{fig:winglist02}
\end{center}
\end{figure}

\begin{figure}[H]
\begin{center}
\includegraphics[width=15cm]{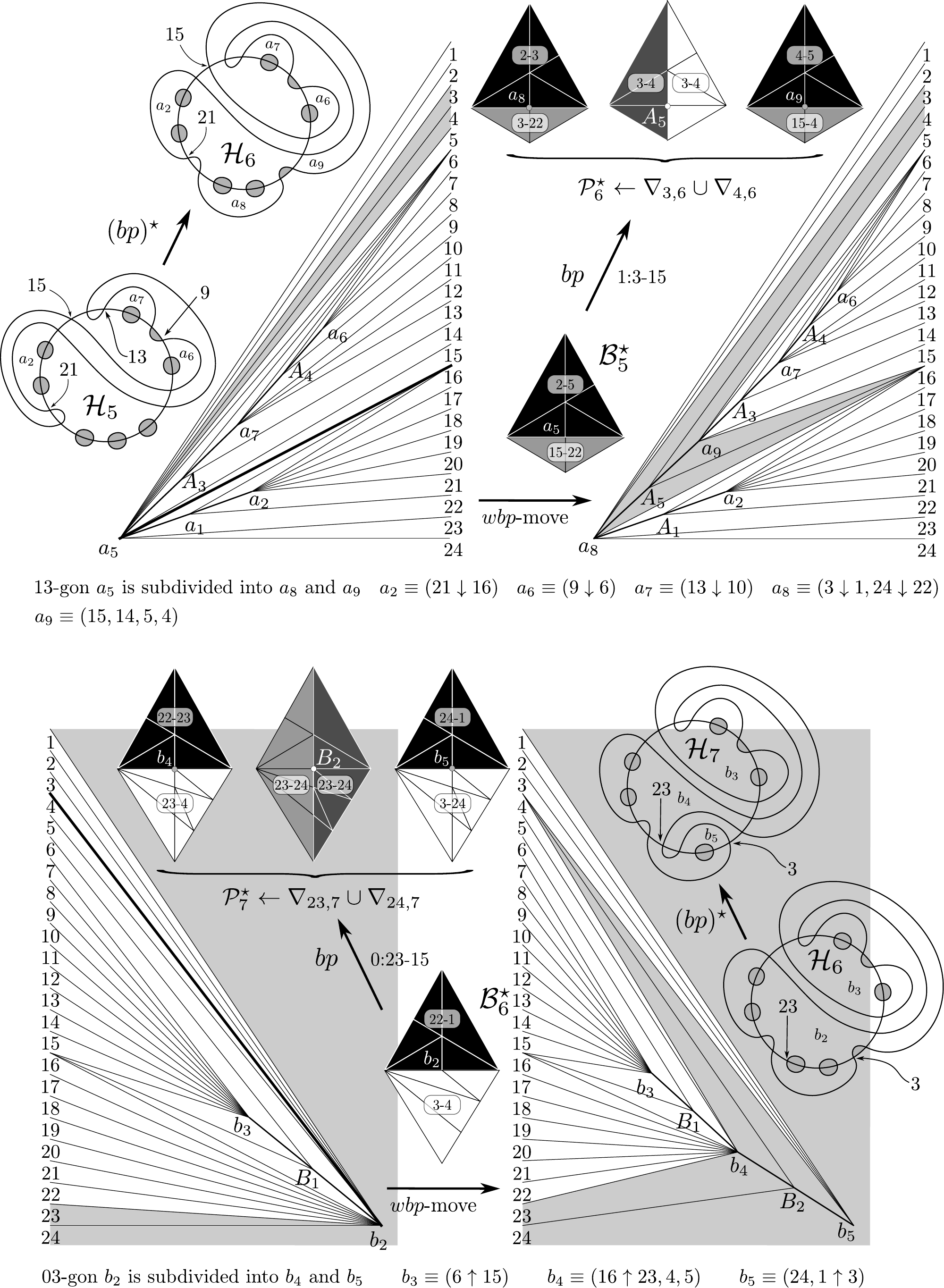}
\caption{\sf
In the first $wbp$-move a left strut is {\em extended based on $a_5$}. 
In all the figures of this $r^{24}_5$-example (except for the $\mathcal{W}$ at the last one) 
we have used Tutte's barycentric
method \cite{tutte1963dg,colin2003tutte} to obtain the embedded final struts. 
However, except for $\mathcal{W}$, only the combinatorics of the embeddings (the rotations)
are needed to encode the struts in an implementation of the algorithm.
In the second $wbp$-move one of the two shaded triangles
is in the outside. Despite this special case, the rotation manifestation of the 
move behaves as usual.
}
\label{fig:winglist03}
\end{center}
\end{figure}

\begin{figure}[H]
\begin{center}
\includegraphics[width=14.7cm]{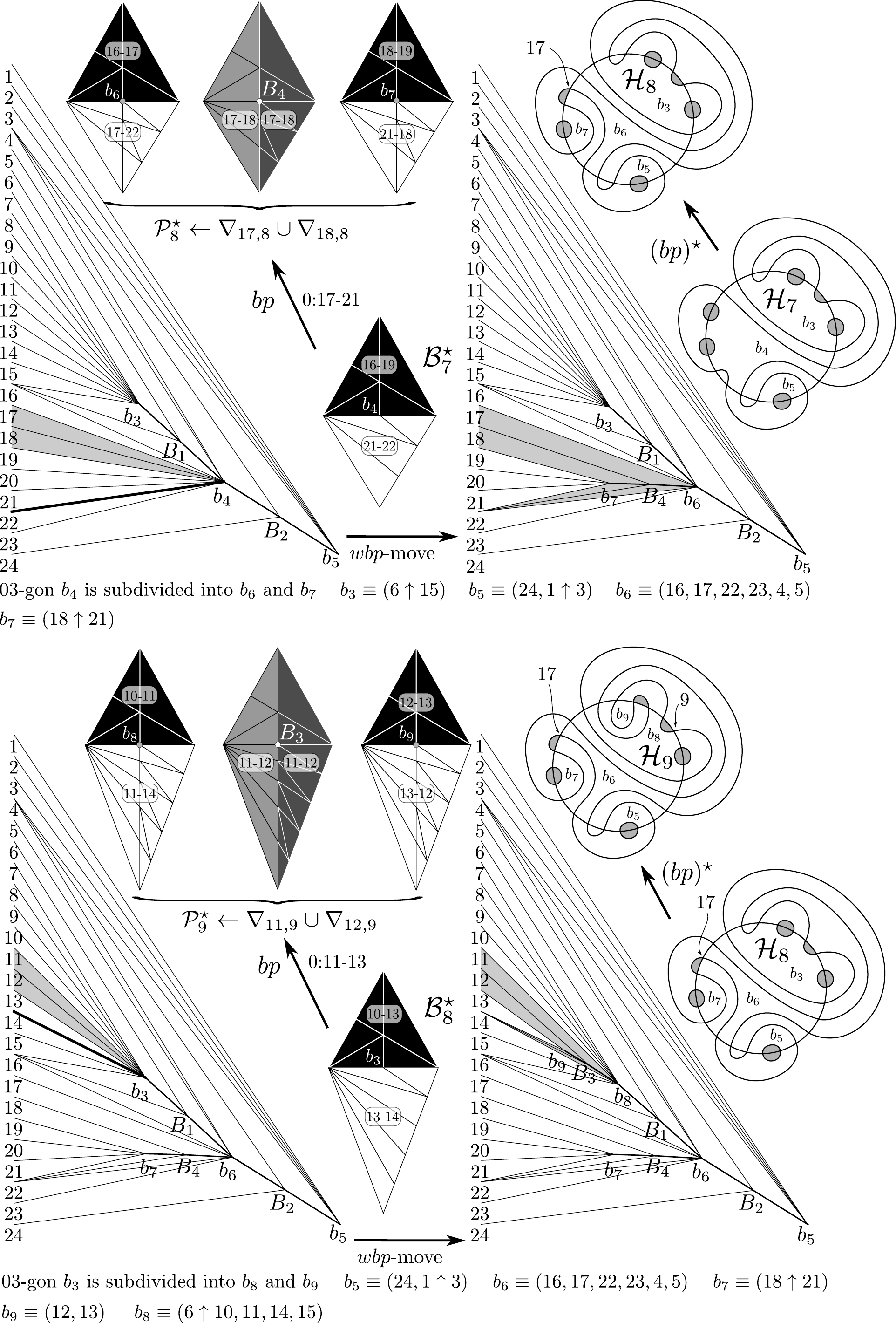} \\
\caption{\sf The first $wbp$-move induces a bifurcation on the nervure of a
right wing based on $b_4$. 
The second $wbp$-move produces an extension based on $b_3$ 
of the nervure of the final right
wing of the first $wbp$-move.
}
\label{fig:winglist04}
\end{center}
\end{figure}

\begin{figure}[H]
\begin{center}
\includegraphics[width=15cm]{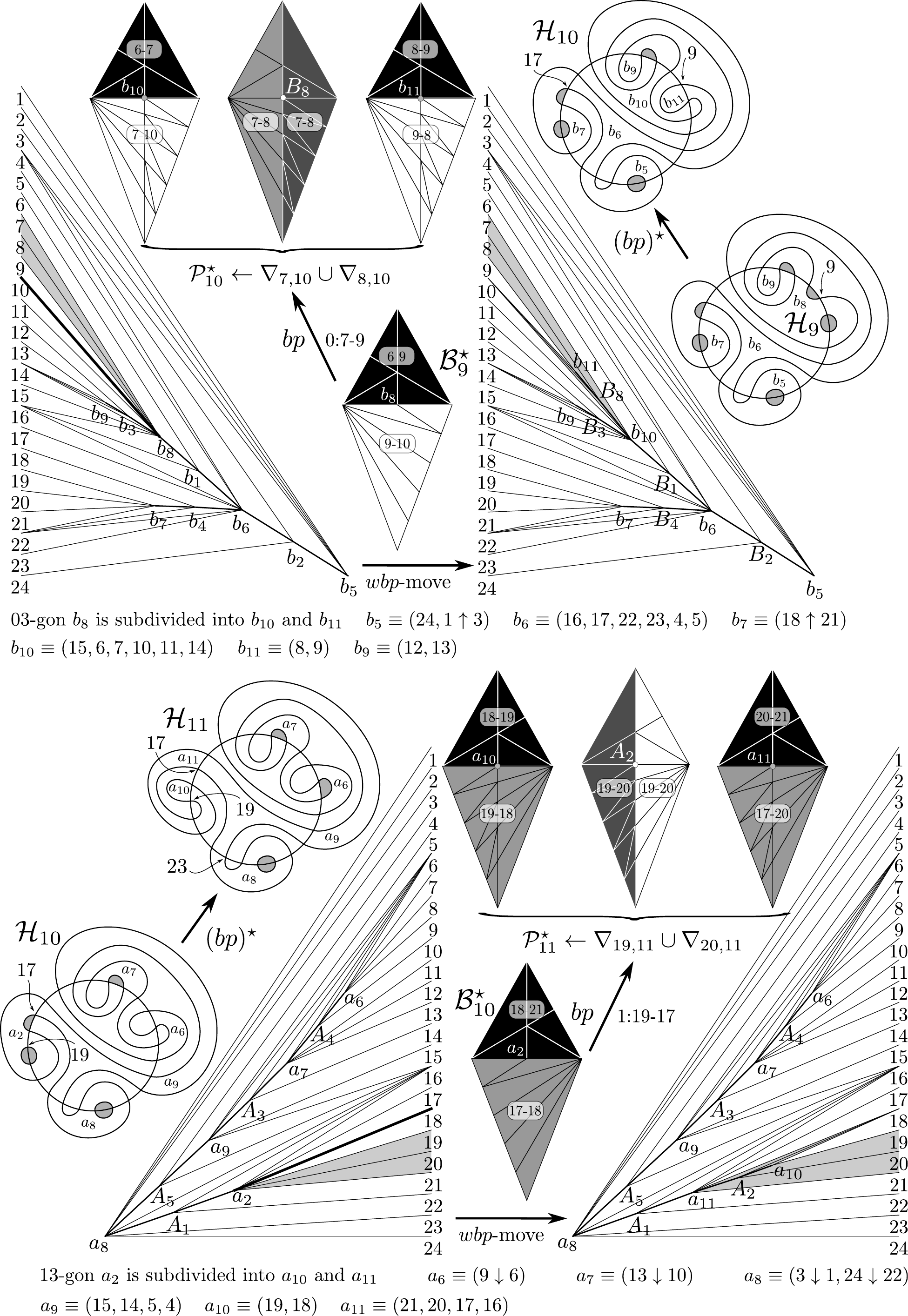} \\
\caption{\sf The first $wbp$-move induces another bifurcation on the nervure of a
right wing based in $b_8$. The second $wbp$-move 
produces an extension of the nervure of a left
wing based on $a_2$.}
\label{fig:winglist05}
\end{center}
\end{figure}

\begin{figure}[H]
\begin{center}
\includegraphics[width=14.5cm]{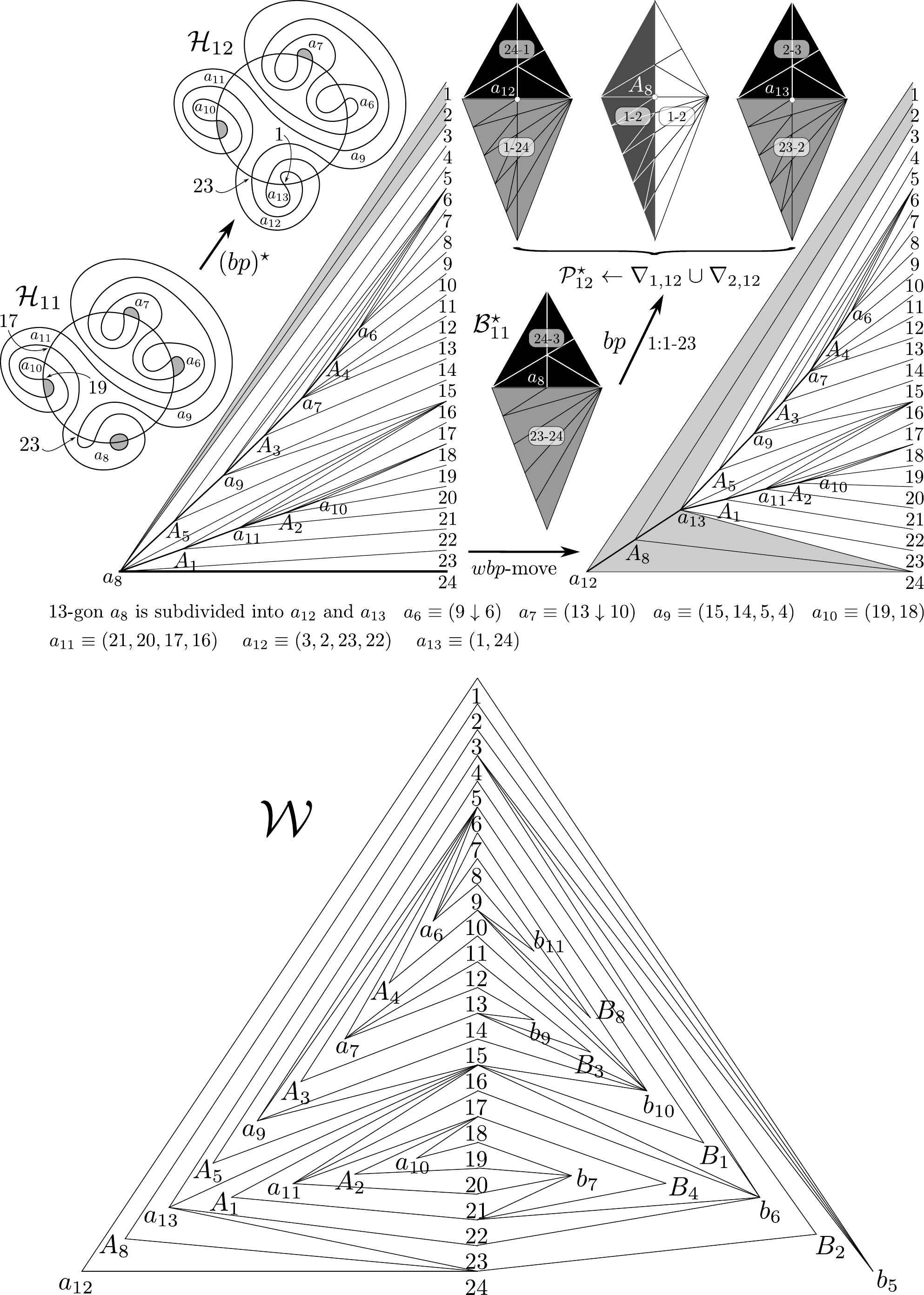} \\
\caption{\sf The globally last $wbp$-move is the extension based at $a_8$ 
(corresponding to a 13-gon which breaks into two $a_{12}$ and $a_{13}$). The bottom part of
the figure depicts the final pair of wings $\mathcal{W}$. 
This is the first time that the geometry of the embedding needs to be specified.
Up to this point we need only the struts given combinatorially.
To obtain the above embedding for  $\mathcal{W}$ we have used the deterministic 
linear algorithm explained in Subsection \ref{subsec:diamondW}. 
}
\label{fig:winglist06}
\end{center}
\end{figure}

\section{Details of the whole construction}
\subsection{First phase: from a $J^2$-gem $\mathcal{H}$ to a bloboid $\mathcal{B}$}

A $k$-dipole $\{u,v\}$ involves color $i$ if it 
if there is an edge of color $i$ linking $u$ to $v$.

Let $H$ be the $J^2$-gem formed by the two Jordan curves $X$ and $Y$.
A {\em $(X,Y)$-duet} in $H$ is a pair of crossings which are consecutive in $X$ and in $Y$.
A {\em $(X,Y)$-trio} in $H$ is, likewise, a triple of crossings that 
are consecutive in $X$ and in $Y$.

\begin{lemma}
\label{lem:trio}
Let $H$ be a $J^2$-gem with $2n\geq4$ vertices. Then $H$ has a $(X,Y)$-trio. 
\end{lemma}
\begin{proof}
By the Jordan theorem $H$ has a $(X,Y)$-duet $D$. If $n=2$ then $H$ has clearly a 
$(X,Y)$-trio establishing the basis of the induction.
Suppose $H$ has $2n\geq 4$ vertices. It has a $(X,Y)$-duet $D$.
If $D$ extends to a trio, then we are done.
Otherwise slightly deform $Y$ to miss $D$. The resulting
$J^2$-gem $H'$ has $2n-2$ crossings and by induction hypothesis $H'$ has a $(X,Y)$-trio $T$,
which is present in $H$, establishing the Lemma.
\end{proof}

\begin{proposition}
\label{prop:startingwith}
Starting with a $J^2$-gem $\mathcal{H}$ with $2n$ vertices we can arrive to an
$n$-bloboid $\mathcal{B}$ by means of $n-1$ operations which 
thickens a 2-dipole involving color 2 into a 3-dipole, where the new edge is of color 0 or color 1,
producing a sequence of $J^2B$-gems each inducing $\mathbb{S}^3$,
$$(\mathcal{H}=\mathcal{H}_{n},\mathcal{H}_{n-1},\ldots, 
\mathcal{H}_{2},\mathcal{H}_{1}=\mathcal{B}).$$

\end{proposition}
\label{prop:backsequence}
\begin{proof}
The proof is by backward induction. 
For $\ell=n$ we have $\mathcal{H}_{n}=\mathcal{H}$ and
so it is a $J^2B$-gem, establishing the basis of the induction. Assume that
$\mathcal{H}_\ell$ is a $J^2B$-gem.
For $\ell>1$, let $\mathcal{H}'_\ell$ denote  $\mathcal{H}_\ell$ after canceling the blobs.
By Lemma \ref{lem:trio} $\mathcal{H}_\ell'$ has a $(X,Y)$-trio $(x,y,z)$. Thus $y$
is incident to two 2-dipoles. One of these dipoles, call it $D$,
involves color 2, the other involves color 3.
Take the one involving color 2, name it $D$. Put back the blobs over the edges of color 3.
So, $D$ is present in $\mathcal{H}_\ell$. The colors involved in 
$D$ are $0$ and $2$ or $1$ and $2$. In the first case we use color 1 to thicken $D$ and
in the second we use color $0$ for the same purpose.
This defines the $J^2B$-gem $\mathcal{H}_{\ell-1}$, which
establishes the inductive step. In face of Proposition \ref{prop:j2shere} and from
the fact that thickening dipoles on gems produce gems inducing the same manifold, 
every member of the sequence induces $\mathbb{S}^3$.
\end{proof}


\subsection{Second phase: colored abstract complexes, their wings, nervures}
The second phase starts with an easy task, namely to define the dual of the bloboid,
named $\mathcal{H}_1^{\star}$. We get this first term in an embedded form. 
The others $\mathcal{H}_2^{\star}$, \ldots, $\mathcal{H}_n^{\star}$,
are, at this stage, obtained by slight modifications of the ancestor, but only in 
an abstract combinatorial way. In doing so we get the minimum 
level of refinement in the PL2-faces required, so that latter, the levels are sufficient
for a geometric PL-embedding in $\mathbb{R}^ 3$ which we seek.

\subsubsection{Primal and dual correspondence between the gem and the colored complex}
There is a simple topological interpretation between primal and dual complexes, given in 
\cite{lins1995gca} pages 38, 39. Let us take a look 
at this interpretation in our context. This will help 
 to understand the $PL$-embedding $\mathcal{H}_m^\star$.
In what follows the $k$ in PL$k$-face means the dimension $k\in\{0, 1, 2,3\}$ of the PL-face.

\begin{itemize}
 \item [i.] 
a vertex $v$ in $G$ $\rightleftharpoons$ a 
solid PL-tetrahedron or PL3-face,
denoted by $\nabla_v$ in the dual of the gem whose
PL0-faces are labeled $z_0$, $z_1$, $z_2$ e $z_3^v$; 
in this work, it is enough to work with the 
boundary of a PL3-face; this is topologically a sphere
$\mathbb{S}^2$ with four PL2-faces one of each color; the 3-simplices forming
a PL3-face need not be explicitly specified;
 \item [ii.] an $i$ colored edge $e_i$ in $G$ $\rightleftharpoons$ a set 
of $i$-colored 2-simplices defining 
a PL2$_i$-face in the dual of the gem;
 \item [iii.] a bigon $B_{ij}$ using the colors $i, j$ in 
$G$ $\rightleftharpoons$ a set of  1-simplices $b_{ij}$
 in $\mathcal{H}_n^\star$ defining a 
PL1$_{ij}$-face;
 \item [iv.] an $\overline{i}$-residue $V_i$ in $G$ 
$\rightleftharpoons$ a 0-simplex in $\mathcal{H}_n^\star$ defining a PL0$_i$-face.
\end{itemize}

\subsubsection{Defining $\mathcal{H}_1^\star$ and the colored abstract PL-complexes: 
$\mathcal{H}_2^{\star}$, \ldots, $\mathcal{H}_n^{\star}$}

We define the combinatorial 2-dimensional PL complex $\mathcal{H}_1^{\star}$ as follows.

\begin{figure}[H]
\begin{center}
\includegraphics[scale=1]{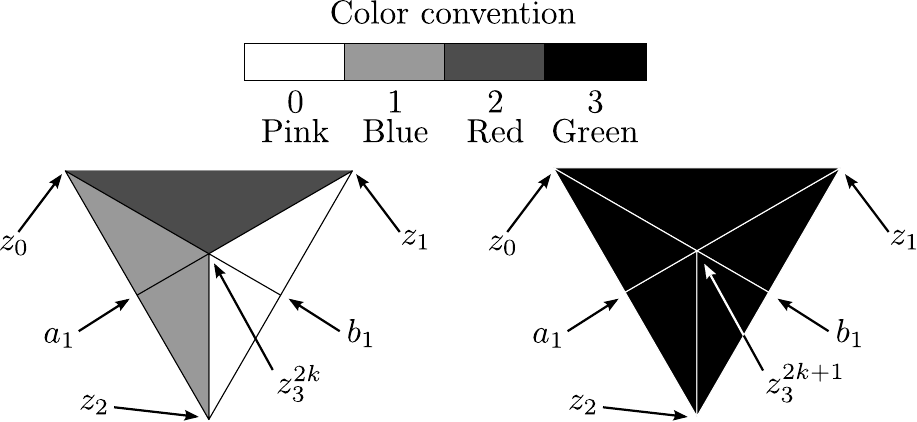}
\caption{\sf PL2-faces of $\mathcal{H}^\star_1$: the figure is an abbreviation of a stack
of tetrahedra, where $k = 0,1,\ldots,n-1.$ The 0-simplices 
$z_3^j$ are defined as $z_3^j=(0,0,2n-j)$, $1\le j \le 2n$. For even $j=2k$,
there are five simplices incident to $z_3^j$: two 0-colored, two 1-colored and 1 2-colored.
For odd $j=2k+1$, the five 2-simplices incident to $z_3^j$ are all 3-colored.
The 0-simplices $z_0$, $z_1$ and $z_2$ are positioned in clockwise order 
as the vertices of an equilateral triangle 
of side 1 in the 
$xy$-plane so that $z_0z_1$ is parallel to the $x$-axis and the center of the triangle
coincides with the origin of an $\mathbb{R}^3$-cartesian system. 
The 0-simplex $a_1$ is $\frac{z_0+z_2}{2}$. The 0-simplex
$b_1$ is $\frac{z_2+z_1}{2}$. Note that, in general, the PL3-faces are given
by their boundary. We never use 3-simplices explicitly to triangulate the PL3-faces.
From our construction, however, it will be clear that this is possible to achieve
without spurious intersections among the 3-simplices.
}
\label{fig:dualPL2}
\end{center}
\end{figure}

We detail the connection of the $J^2$-gem and its dual. In particular we 
use the unique 23-gons of it to provide labels $1,2,\ldots,2n$ in the cyclic 
order of the 23-gon. This labellings correspond to PL3-faces of the 
$\mathcal{H}_1^{\star}$ and will be maintained for the PL3-faces 
of the whole remaining sequence $\mathcal{H}_2^{\star}$, 
\ldots, $\mathcal{H}_n^{\star}$. This invariance is a dual 
manifestation of the fact that in
the thickening of dipoles the labels of the vertices preserved.
Suppose $u$ is an odd vertex of the $J^2$-gem, $u'=u-1$, $v=u+1$ and $v'= v+1$.
The dual of a $\overline{3}$-residue is $z_3^j$ where $j$ is even. 
When $j$ is odd, then $z_3^j$ is a $0$-simplex in the middle of a PL2$_3$-face,
incident to five 2-simplices of color 3.
The dual of the 03-gon is the PL1-face formed by the pair of 1-simplices $z_1b_1$ 
and $b_1z_2$. The dual of the 13-gon is the PL1-face formed by the pair of 
1-simplices $z_0a_1$ and $a_1z_2$. 
The dual of the 23-gon is the PL1-face formed by the 1-simplex $z_0z_1$.
The dual of the 01-gon relative to vertices $u$ and $v$ is the 1-simplex $z_2z_3^v$. 
The dual of the 02-gon relative to vertices $u$ and $v$ is the 1-simplex $z_1z_3^v$.
The dual of the 12-gon relative to vertices $u$ and $v$ 
is the 1-simplex $z_0z_3^v$.
The dual of a 3-colored edge $u'u$ is the image of PL2$_3$-face with odd index $u$ in the vertices.
The dual of an $i$-colored 
edge $uv$ with $i\in\{0, 1, 2\}$ is the PL2$_i$-face with even index $v$.

\subsubsection{Primal and dual $bp$-moves}
\label{subsec:primaldualmoves}
Before presenting $\mathcal{H}_m^\star$, $1< m \le n,$ and its embeddings, we need to understand
the dual of the $(pb)^\star$-move and its inverse. In the primal, to apply a $(pb)^\star$-move, we need
a blob and a 0- or 1-colored edge. The dual of this pair is the {\em balloon:} the \index{balloon's head} {\em balloon's head} is
the dual of the blob; the {\em balloon's tail} \index{balloon's tail} is the dual of the $i$-edge. To make it easier to
understand, the $(pb)^\star$-move can be factorable into a 3-dipole move followed by a 2-dipole move, 
so in the dual, it is a smashing of the head of the balloon followed by the pillow move 
described in the book \cite{lins1995gca}, page 39.
This composite move is the {\em balloon-pillow move} \index{balloon-pillow move} or \index{bp-move} {\em bp-move}.
Restricting our basic change in the colored 2-complex 
to $bp$-moves we have nice theoretical properties which are responsible for avoiding an exponential
process. 

\begin{figure}[H]
\begin{center}
\includegraphics[width=15cm]{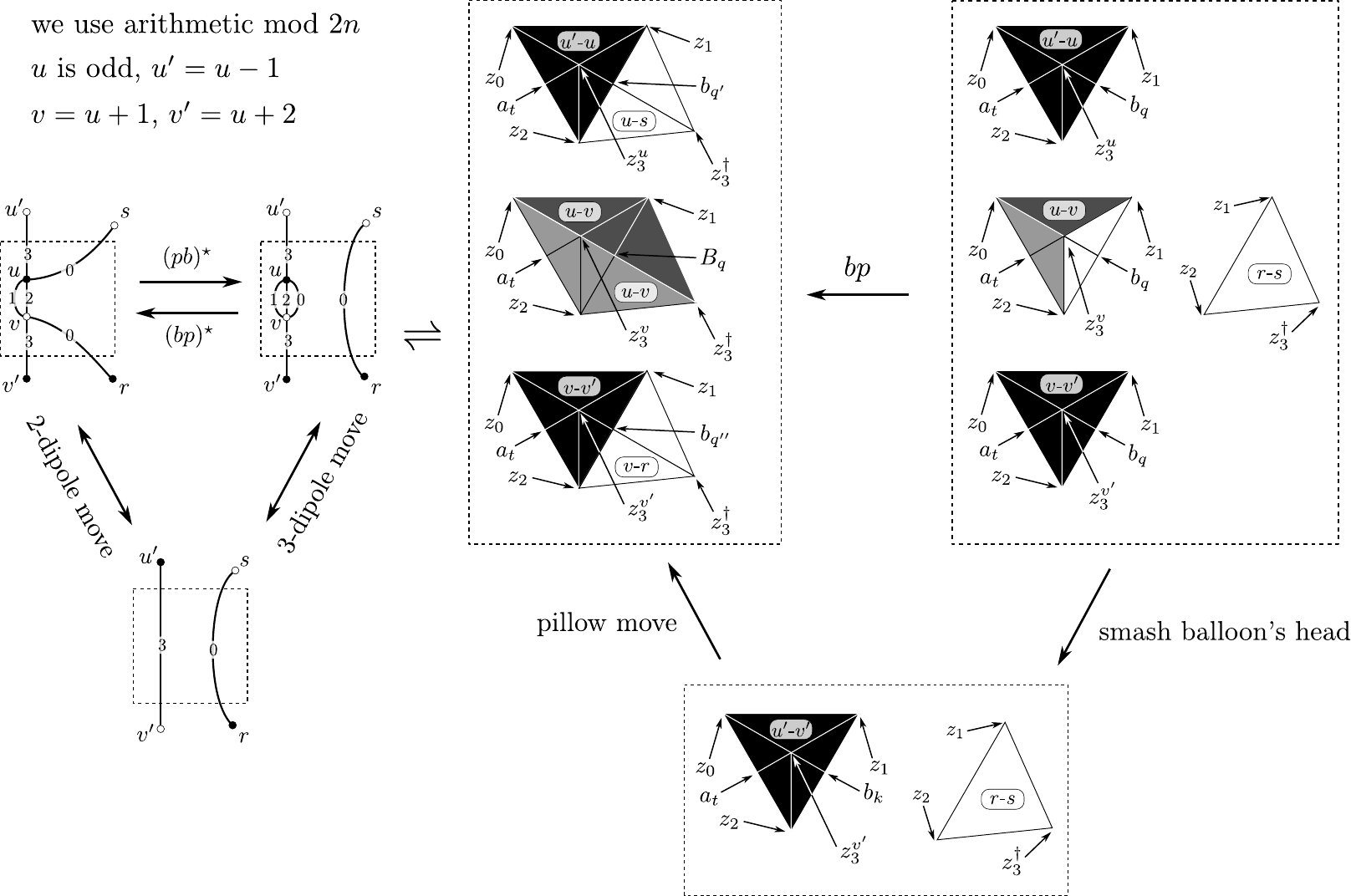}
\caption{\sf Primal and dual $bp$-moves: in what follows we describe the $bp$-move assuming that the balloon's tail is 0-colored
using a generic balloon's tail, of which we just draw the contour. The other case, color 1, is similar.
(1) if the image of $v_5^u$ and $v_5^v$ is $b_q$, create two 0-simplices $b_{q'}$ and $b_{q''}$, 
define the images of $v_5^u$ and $v_5^{v'}$ as $b_{q'}$ 
and $b_{q''}$ and change the label of the image of $v_5^{v}$ from $b_q$ to $B_q$;
(2) make two copies of the PL2$_0$-face, if necessary, 
refine each, from the middle vertex of the segment $z_2z_1$ to the third vertex $z_3^\dagger$, where $\dagger=j,$ for an adequate height $j$;
(3) change the colors of the medial layer of the pillow as specified by the 
dual structure, namely by the current $J^2B$-gem.
}
\label{fig:pillowother2NOVO}
\end{center}
\end{figure}

\subsubsection{Types of PL2-faces}

\begin{proposition}
\label{prop:allkinds}
Each PL2-face of the combinatorial simplicial complex $\mathcal{H}_m^\star$,
$1\le m \le n$, is isomorphic to 
one in the {\em set of types of triangulations} \index{types of triangulations}
$$\{G, P_{2k-1}, P_{2k-1}', B_{2k-1}, B_{2k-1}', R_{2k-1}^b, R_{2k-1}^p \ | 
\ k\in \mathbb{N} \},$$
described in Fig. 
\ref{fig:allkinds3}, where the index means the number of 
edges indicated and is called the {\em
rank of the type}. \index{rank of the type} Moreover, 
the PL2-faces that appear, as duals of the gem edges, 
have the minimum number of 2-simplices.
\end{proposition}

\begin{figure}[H]
\begin{center}
\includegraphics[width=13.1cm]{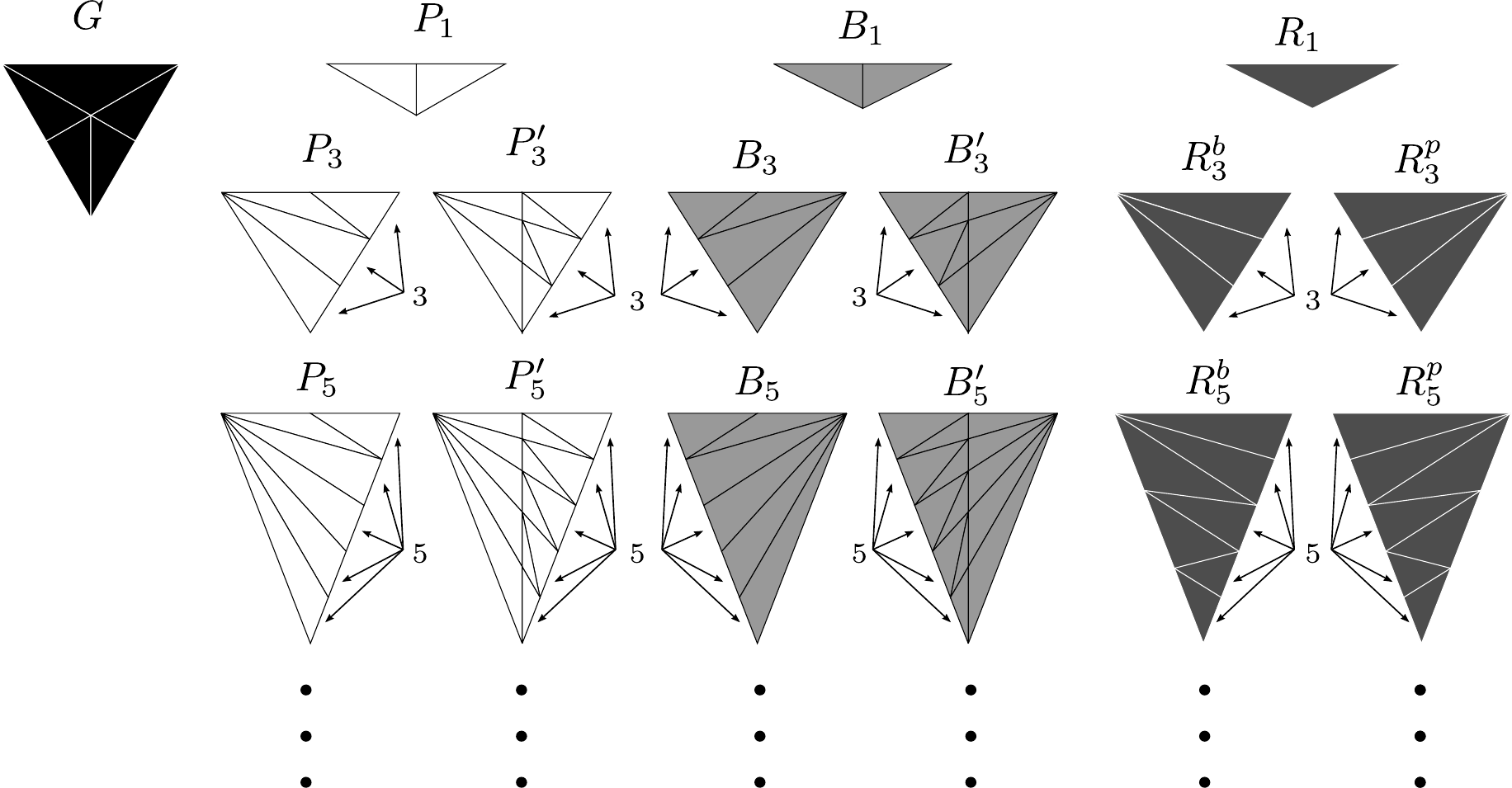}
\caption{\sf All kinds of PL2-faces that we use: 
the choice of the letters $P, B, R, G$ 
comes from the colors $0=(P)ink$, $1=(B)lue$, 
$2=(R)ed$ and $3=(G)reen$. Define
$R_{2k-1}^b$ as the PL2$_2$-face which is inside
the pillow neighboring a PL2$_1$-face. Similarly  $R_{2k-1}^p$ is a PL2$_2$-face which is inside
the pillow neighboring a PL2$_0$-face.These PL2-faces 
are for now abstract combinatorial triangulations that
have the correct level of refinement so as to become PL-embedded into $\mathbb{R}^ 3$.
}
\label{fig:allkinds3}
\end{center}
\end{figure}

\begin{proof}
We need to fix a notation for the head of the balloon. Instead of drawing all the PL2-faces 
of the head, 
we just draw one PL2$_3$-face and put a label $u'$-$v'$. If the balloon's tail, is 
of type $P_1$, by applying a $bp$-move we can see at Fig. \ref{fig:pillowother5}
\begin{figure}[!htb]
\begin{center}
\includegraphics[width=12cm]{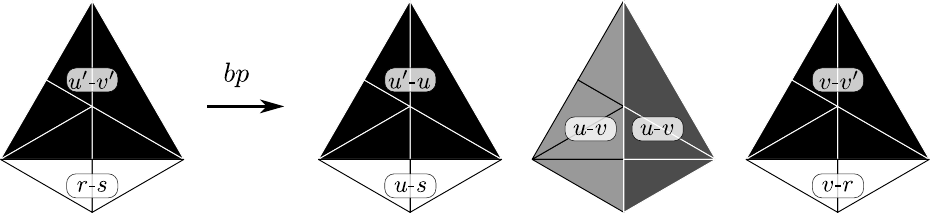}
\caption{\sf A $bp$-move with the balloon's tail of type $P_1$.}
\label{fig:pillowother5}
\end{center}
\end{figure}
that we get a PL2$_1$-face of type $B_3$ and a PL2$_2$-face of type $R_3^b$. 
The others PL2-faces are already known. 
If the balloon's tail, is of type $B_3$, by applying a $bp$-move, we need to refine the tail 
and the copies, otherwise we would not be able to build a pillow
 because some 2-simplices would be collapsed, so
we get two PL2$_1$-faces of type $B_3'$, one PL2$_0$-face of type $P_5$ and a PL2$_2$-face
 $R_5^p$.  The others PL2-faces are already known.

\begin{figure}[!htb]
\begin{center}
\includegraphics[width=12cm]{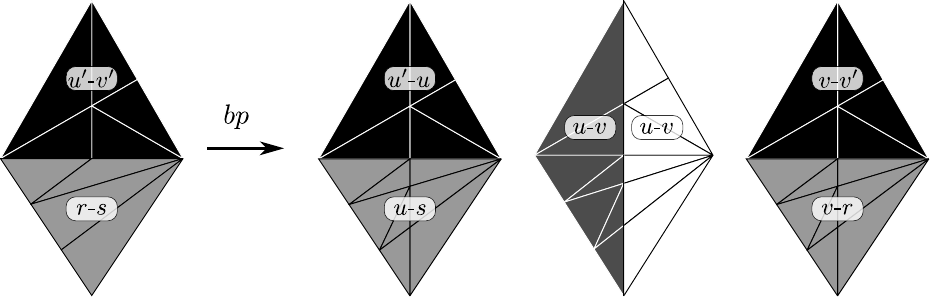}
\caption{\sf A $bp$-move with the balloon's tail of type $B_3$.}
\label{fig:pillowother3}
\end{center}
\end{figure}

In what follows given $X \in \{ P_{2k-1}', B_{2k-1}' \}$ denote by 
$\overline{X}$ the copy of $X$ which is a PL2-face of the PL-tetrahedra 
whose PL2$_3$-face is below the similar PL2$_3$-face of the other PL-tetrahedra 
which completes the pillow in focus. 
In face of these conventions, if balloon's tail is of type
\begin{itemize}
 \item $P_{2k-1}$, then by applying a $bp$-move, we get types
$P_{2k-1}', \widehat{P}_{2k-1}', B_{2k+1}, R_{2k-1}^b $
 \item $P_{2k-1}'$, then by applying a $bp$-move, we get types 
$P_{2k-1}, B_{2k+1}, R_{2k-1}^b$ 
 \item $B_{2k-1}$, then by applying a $bp$-move, we get types
$B_{2k-1}',  \widehat{B}_{2k-1}', P_{2k+1}, R_{2k-1}^p$
 \item $B_{2k-1}'$, then by applying a $bp$-move, we get types 
$B_{2k-1}, P_{2k+1}, R_{2k-1}^p$ 
\end{itemize}
The necessary increasing in the ranks of the types of faces shows that the rank of 
each face is at least the one obtained. It may cause a surprise the fact that these
ranks are enough to make the PL-embedding geometric into $\mathbb{R}^3$.

\end{proof}

It is worthwhile to mention, in view of the above proof, 
that each PL2-face is refined at most one time. 
So, if $X$ is a type of PL2-face, $X'$ is its refinement, then
$X''=X'$. This idempotency is a crucial property inhibiting 
the exponentiality of our algorithm and enables a quadratic bound.

\subsubsection{Quadratic bounds on the number of simplices}

\begin{corollary}
\label{theo:thequadratic}
The quadratic expressions $$3n^2-5n+9,~~~ 11n^2-17n+21,~~~ 8n^2-10n+12$$
 are upper bounds for the numbers of 0-simplices, 1-simplices and 2-simplices
of the colored 2-complex $\mathcal{H}_{n}^\star$ induced by a resoluble
gem $G$ with $2n$ vertices.
\end{corollary}
\begin{proof} 
We prove the first bound, on 0-simplices; the other are similar: in the worse case,
the increase of simplices is a linear function on the rank of the current PL2-face,
and to get the final number we sum an arithmetic progression.

We detail the strategy for 0-simplexes. Note that $\mathcal{H}_{1}^\star$ has exactly 
$z_0, z_1, z_2, a_1, b_1$ and $z_3^j, j\in \{1,\ldots,2n\}$ as 0-simplices, which is $2n+5$
0-simplices. In the first step, the balloon's tail has to be of type $P_1$ or $B_1$, 
so by applying a $bp$-move, we get two new 0-simplices.
In second step, the worst case is when the balloon's 
tail is of type $P_3$ or $B_3$, generated by last $bp$-move, 
so we add $6\times1+2=8$ to the number of 0-simplices in the upper bound.
In step $k$ we note that the worst case is when we use the greatest ranked PL2-face generated by 
last $bp$-move, therefore the balloon's tail has to be of type $P_{2k-1}$ or $B_{2k-1}$
and we add $6\cdot(k-1)+2$ 0-simplices. 
By adding the number of 0-simplices created by $bp$-moves from step 1 until step $k$ 
we get  $3k^2-k$ 0-simplices. Since the number of steps is $n-1$, 
and we have at the beginning $2n+5$ 0-simplices, we have that $3n^2-5n+9$ is 
an upper bound for the number of 0-simplices. 
\end{proof}

\subsubsection{Detailing the combinatorics of the wings, nervures and struts}

An embedding of a graph into an oriented surface is combinatorially encoded
as a {\em rotation at the vertices} or simply a {\em rotation}. A rotation is the set of cyclic orderings
of the edges around the vertices (induced by the surface) so that each edge appears exactly twice, 
one with each orientation. We encode the struts as rotations. Only the last one needs to be
dealt with geometrically.

At some point in our research it became evident that what was 
needed to obtain the embedded PL-complex $\mathcal{H}_n^\star$ 
was a proper embedding into $\mathbb{R}^3$ of two special sequences 
of 0-simplices $\mathcal{A}=(a'_1, a'_2,\ldots,a'_f)$ and $\mathcal{B}=(b'_1, b'_2,\ldots,b'_g)$,
where $f+g=2n$. Each $a'_i \in\{a_i,A_i\}$ and each $b'_i \in\{b_i,B_i\}$. 
This terminology for the 0-simplices is obtained recursively
and detailed shortly in this section. For now we just say that all 
other 0-simplices not in $\mathcal{A}\cup\mathcal{B}$ of the colored complexes $\mathcal{H}^\star_m$
are obtained by bisections of segments linking previously defined points. 
It came as a surprise to discover
that the apparently difficult 3D problem of 
positioning $\mathcal{A}\cup\mathcal{B}$ could be 
reformulated as a plane problem with an easy solution,
via a linear unique solution algorithm. That is the role of the wings, nervures and
struts associated to the colored complexes.

\begin{figure}[H]
\begin{center}
\includegraphics[width=15.4cm]{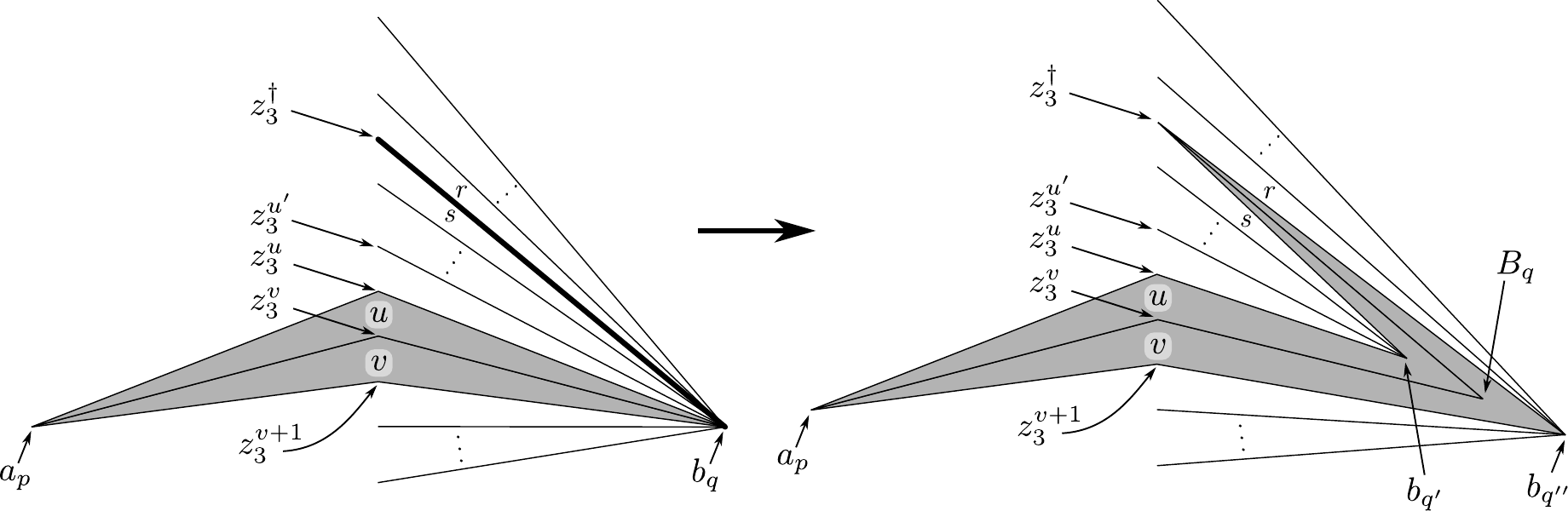}
\caption{\sf Generic $wbp$-move: balloon's head section is painted in gray, 
and the part of balloon's tail that is intersecting the appropriate
semi-plane $(\Pi'')$ is depicted as a {\em thick edge}. \index{thick edge}}
\label{fig:controlmaps01}
\end{center}
\end{figure}

\begin{figure}[H] 
\begin{center}
\includegraphics[width=15 cm]{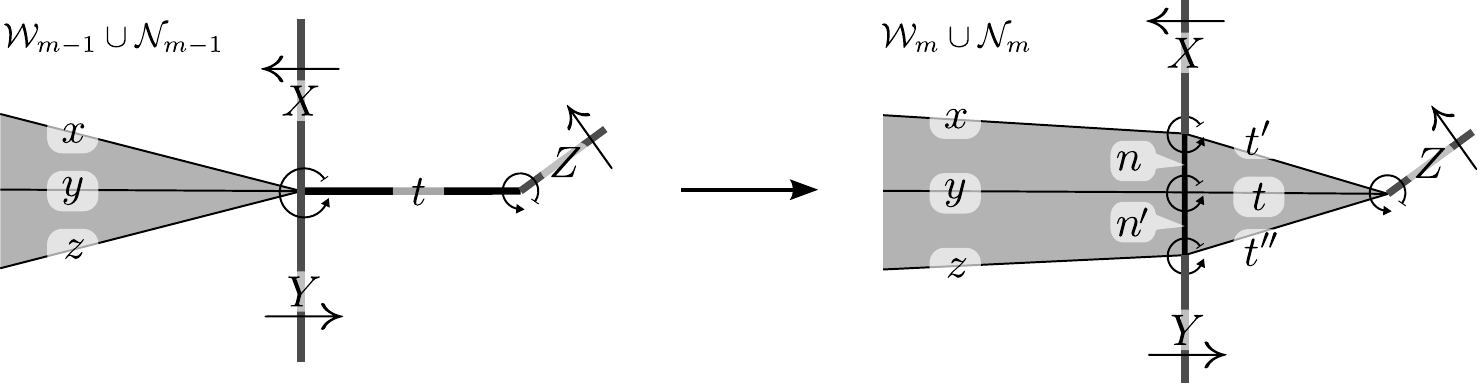}
\caption{\sf The effect of the general $wbp$-move
in the combinatorial strut: the {\em star} \index{star} of a vertex 
of a graph embedded in an orientable
surface (in our case the plane) is the counterclockwise cyclic sequence of 
edges incident to the vertex (such an ordering is induced by the surface).
The set of stars is called a { \em rotation} \index{rotation} and has the 
characterizing property that each edge appears twice.
The general case of changing rotation when going from 
$\mathcal{W}_{m-1} \cup \mathcal{N}_{m-1}$ to
$\mathcal{W}_{m} \cup \mathcal{N}_{m}$ is depicted above: 
labels $X,Y,Z$ stand for arbitrary (maybe empty) sequences of edges;
the vertex $XxyzYt$ breaks into three $Xxnt'$, $nyn't$ and $n'zYt''$ and
the vertex $Zt$ changes into $Zt'tt''$. Two new vertices and four new edges are created. 
Two of these edges ($n$ and $n'$) are in the nervure $\mathcal{N}_{m}$ and the other two ($t'$ and $t''$) are in the wing $\mathcal{W}_{m}$.
The new rotation completely specifies the topological 
embedding of $\mathcal{W}_{m} \cup \mathcal{N}_{m}$.}
\label{fig:controlmaps01novo2}
\end{center}
\end{figure}

\begin{algorithm}[Obtaining the rotations for the struts, arriving to $\mathcal{S}_n$]
{\sf 
BEGIN: $i=1$; $\mathcal{W}_i=\mathcal{W}'_i \cup \mathcal{W}''_i$ is formed by 
$2n$-edges linking $a_1$ to $Z$ and $2n$ edges linking 
$b_1$ to $Z$; we also have $\mathcal{N}_1=\varnothing$; REPEAT: $i \leftarrow i+1$; 
there are two cases, according to the color of the edges 
which are flipped in the associated thinning of a blob is 1 or 0; 
in the first case the strut that 
changes is the left one 
(while $\mathcal{S}_{i}''=\mathcal{S}_{i-1}''$), 
in the other case it is the right strut that changes
(while $\mathcal{S}_{i}'=\mathcal{S}_{i-1}'$);
perform the appropriate $wbp$-move creating two new vertices $a_{2\ell}, a_{2\ell+1}$ 
in the first case or $b_{2\ell}, b_{2 \ell+1}$ in the second ($2\ell-1$ is the last index
of $\mathcal{A}$ or of $\mathcal{B}$); let $a_q$ (resp. $b_q$) denote the 13-gon (resp. the 03-gon)
which is being subdivided; then we relabel $a_q$ as $A_q$ (resp. $b_q$ as $B_q$); these capital
letter labeled 0-simplices no longer correspond to bigons, while bigon $a_q$ (resp $b_q$) has
been subdivided into $a_{2\ell}, a_{2\ell+1}$ (resp. $b_{2\ell}, b_{2 \ell+1}$); 
also 2 new edges are added when going from $\mathcal{S}_{m-1}$ to $\mathcal{S}_{m}$ according to
the rotation changing of Fig. \ref{fig:controlmaps01novo2}; 
UNTIL $i=n$; Output $\mathcal{S}_n$; END.
}
\end{algorithm}

It is easy to verify that the above algorithm is linear in $n$.
Now we need to make $\mathcal{S}_n'$ rectilinearly embedded into $\Pi'$
and $\mathcal{S}_n''$ rectilinearly embedded into $\Pi''$. 
This is provided in the next subsection.

\subsection{Third phase: a rectilinear PL-embedding for 
$\mathcal{S}_n$ 
and its induced diamond complex $\mathcal{H}_1^\diamond$, by a cone construction}
\label{subsec:diamondW}

A graph is {\em rectilinearly embedded into a plane} if 
the images of their edges are straight
line segments. We will find by a linear algorithm a rectilinear 
embedding for $\mathcal{S}_n$. We do it in the plane and lift it to space
by two simple rotations to position $\mathcal{S}'_n$ in $\Pi'$ and 
$\mathcal{S}''_n$ in $\Pi''$.

\begin{lemma}\label{lem:numberofedges}
 The number of edges of $\mathcal{S}_n$ is $8n-4$.
\end{lemma}
\begin{proof}
The number of edges of $\mathcal{S}_1$ is $4n$. At each of the $n-1$ $bp$-moves
we add 4 new edges.
\end{proof}

We present an algorithm to obtain a rectilinear embedding in the plane $xz$ of
$\mathcal{S}_n$. The algorithm is subdivided into two sub-algorithms:  
the first one
produces the $x$-coordinates of the vertices of $\mathcal{N}_n$ 
and the second one produce the $z$-coordinates
of the vertices of $\mathcal{S}_n$.

\begin{figure}[H] 
\begin{center}
\includegraphics[width=15.5cm]{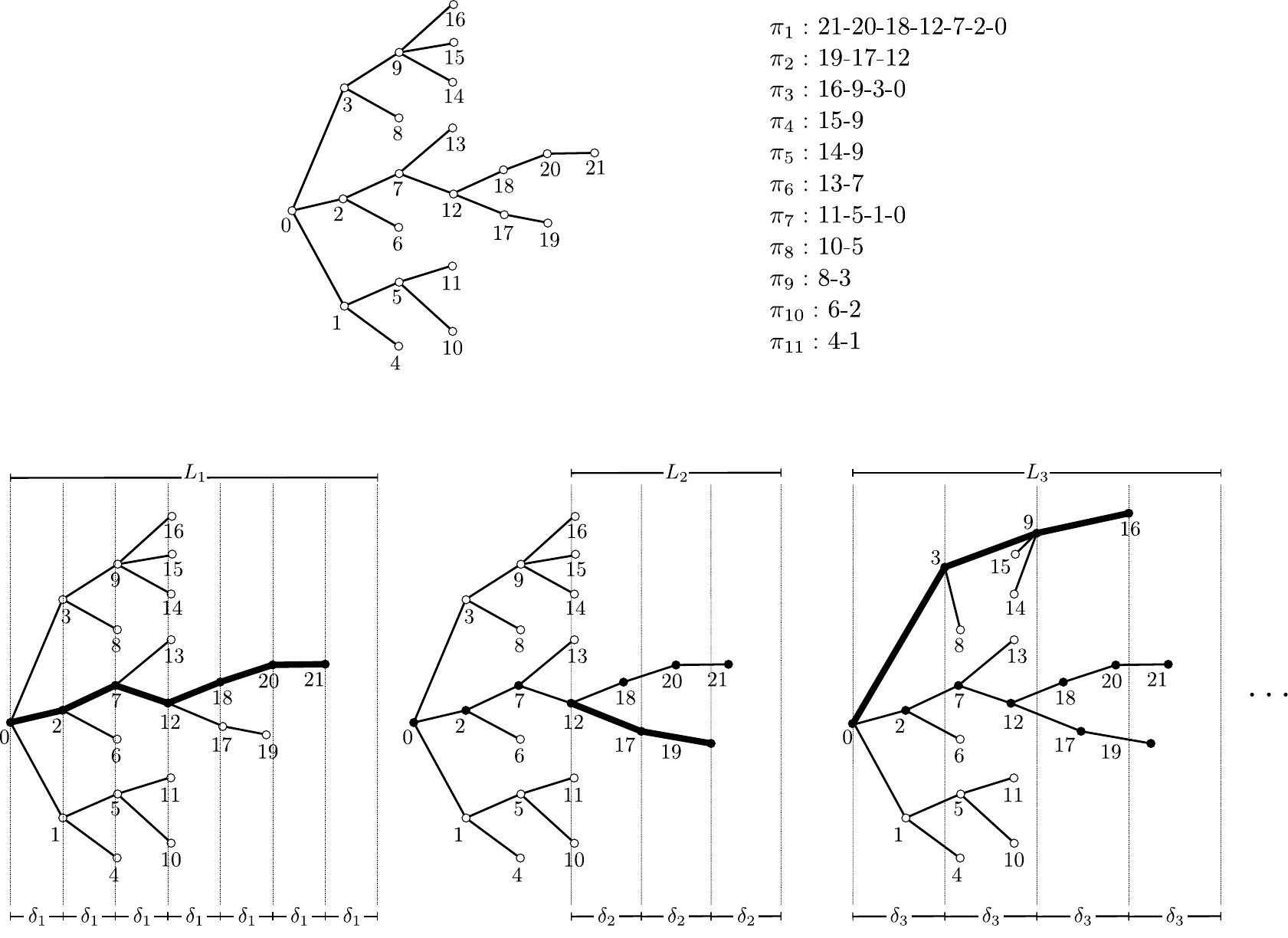}
\caption{\sf A linear sub-algorithm to find the $x$-coordinates of the vertices of the left nervure
$\mathcal{N}'$ (the right nervure $\mathcal{N}''$ is similar). BEGIN: 
find the BFS-number (\cite{cormen2001introduction}) of vertices of the tree $\mathcal{N}'$
with root vertex being the leftmost one of $\mathcal{N}'$. Let the vertices be labeled by its BFS-number.
In the algorithm, we use a partition of the edges of $\mathcal{N}'$
into paths; make all vertices except the root unused; $i=0$; 
make the path partition empty; REPEAT: $i \leftarrow i+1$; take the ancestor path $\pi_i$ 
starting with the highest BFS-numbered vertex not yet used and finishing at the 
first used vertex; put the sequence 
of vertices of $\pi_i$ defining a new member of the edge-path partition;
declare all the vertices in $\pi_i$ as used; 
let $\delta_i$ be $L_i/\lambda_i$, where $L_i$ is the $x$-distance from the first and last 
vertices of $\pi_i$ and $\lambda_i$ is the number of vertices in $\pi_i$; let $\delta_i$ be 
the $x$-distance  between consecutive vertices of $\pi_i$ 
(note that the $x$-coordinate of the last 
vertex of $\pi_i$ has been already defined and so the $x$-coordinates 
of the all the vertices of $\pi_i$ becomes fixed, never to change); 
UNTIL all vertices are used; END. 
}
\label{fig:coordX}
\end{center}
\end{figure}

\begin{figure}[H] 
\begin{center}
\includegraphics[width=14cm]{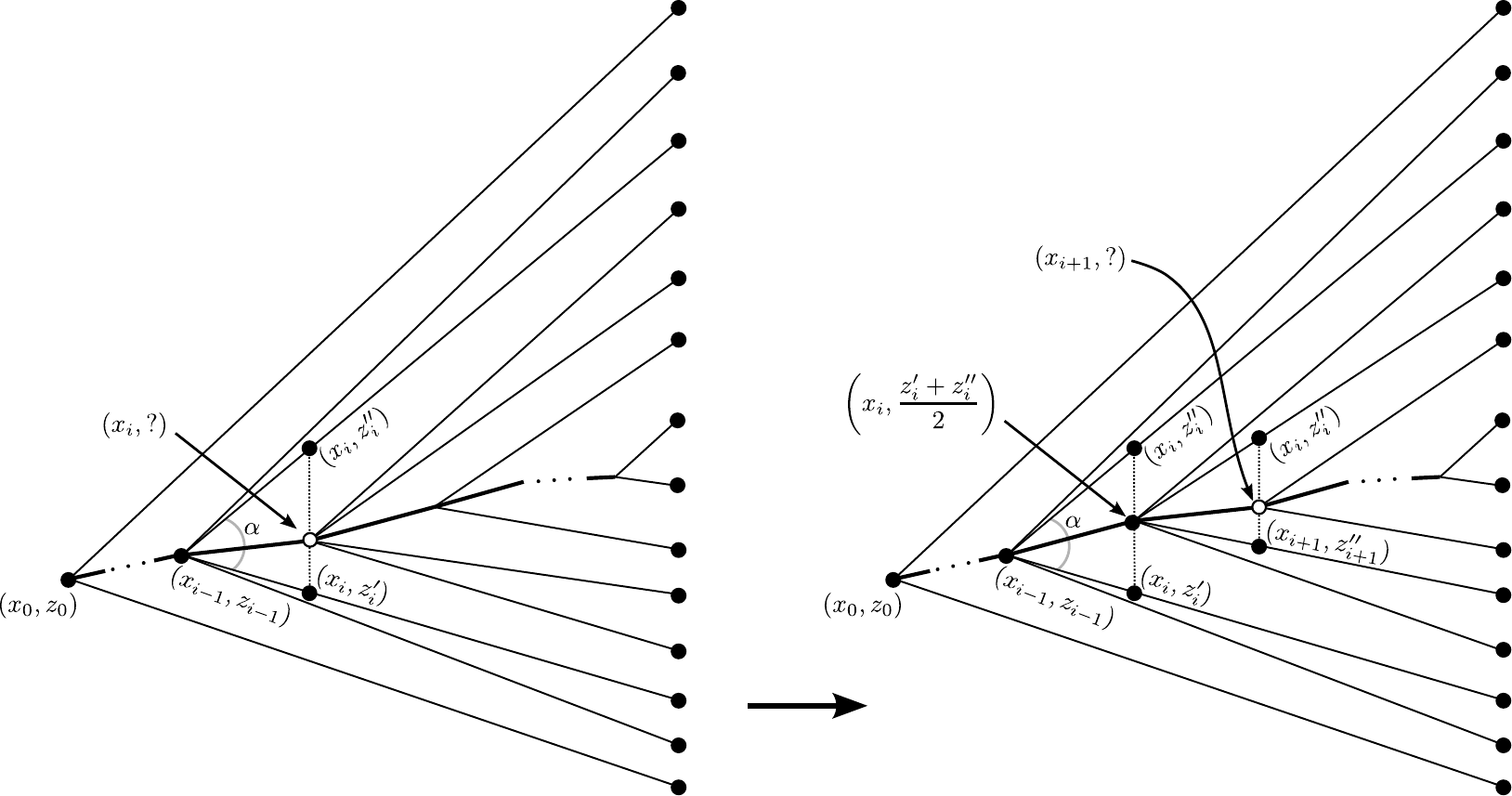}
\caption{\sf A linear sub-algorithm to find the $z$-coordinates of the vertices
for a rectilinear embedding of $\mathcal{W}'$
($\mathcal{W}''$ is similar). BEGIN:
by using the Algorithm of Fig. \ref{fig:coordX} we have already the $x$-coordinates
of all the vertices of $\mathcal{W}'$. Let $\gamma_1, \gamma_2, 
\ldots \gamma_k$ be the sequence of inverses of the 
paths $\pi_i$'s, obtained in the algorithm of Fig. \ref{fig:coordX}; 
let $z_1^1$ be the $z$-coordinate of the first vertex of $\gamma_1$; 
$z_1^1 \leftarrow 0$; $i \leftarrow 0$; REPEAT: $i \leftarrow i+1$;  $j \leftarrow 0$; 
REPEAT: $j \leftarrow j+1$; let $e_{ij}$ be the edge in the 
nervure incident to $\gamma_i^{j}$ (the $j$-th vertex in the path $\gamma_i$)
and $\gamma_i^{j-1}$; 
let $e_{ij}'$ be the edge which succeeds $e_{ij}$ and $e_{ij}''$ 
be the edge that precedes it in the counterclockwise 
rotation of vertex $\gamma_i^{j-1}$; note that the other ends of 
$e_{ij}'$ and $e_{ij}''$ are $z_3^p$ and $z_3^q$ for some $p<q$;
(Obs: in the case of $\gamma_1^1$ it might happen that $e_{11}''$ does not exist; 
in this case define $e_{11}''$ as a virtual edge that links $\gamma_1^1$ to $z_3^{2n}$, 
where $2n$ is the number of vertices of the $J^2$-gem);
let $v$ be the intersection of the line $x=x_i^j$ and the edge $e_{ij}'$;
let $w$ be the intersection of the line $x=x_i^j$ and the edge $e_{ij}''$;
$z_i^j \leftarrow (v+w)/2$; UNTIL $j$ = length of $\gamma_i$; UNTIL $i = k$; 
END.
}
\label{fig:drawLasWing}
\end{center}
\end{figure}

We apply the algorithms of the previous two figures to obtain a rectilinear
embedding of $\mathcal{W}$ in linear time. The connection between 
 $\mathcal{W}$ ,  $\mathcal{W}_1^\diamond$  
 and  $\mathcal{H}_1^\diamond$ is explained in Figure \ref{fig:Wdiamond1}.

\begin{figure}[H]
\begin{center}
\includegraphics[width=14.5cm]{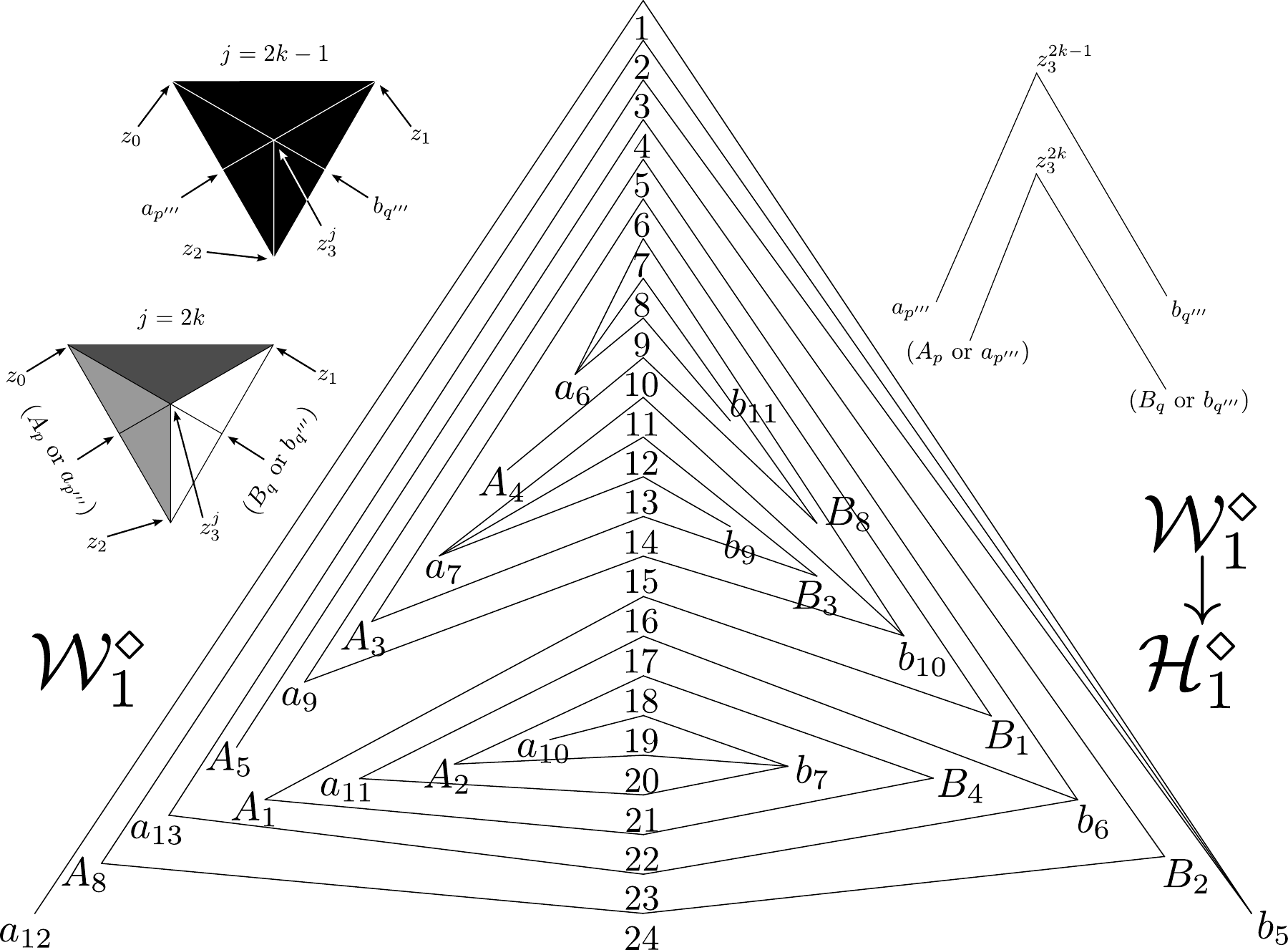} \\
\caption{\sf To get $\mathcal{W}^\diamond_1=\mathcal{W}^{\diamond'}_1 
\cup \mathcal{W}^{\diamond''}_1$ from $\mathcal{W}$ remove
all but two straight line segments emanating from  $z_3^k$, one in each side.
The two segments that survive are the ones 
finishing at the smallest index upper case $A_p$ (when there is a choice)
and the smallest indexed upper case $B_q$ (when there is a choice).
Let $\{x\} \cup Y \subseteq \mathbb{R}^N$, for $1\le N \in \mathbb{N}$. 
The \index{cone} {\em cone} \cite{rourke1982introduction} 
with vertex $x$ and base $Y$, denoted $x \ast Y \subseteq \mathbb{R}^N$,
is the union of $Y$ with all line 
segments which link $x$ to $y \in Y$.
The passage 
$\mathcal{W}_1^ \diamond \rightarrow \mathcal{H}_1^\diamond$ is 
straightforward by a cone algorithm:
for each $e'\in \mathcal{W}_1^{\diamond'}$ add the two 2-simplices $z_0\ast e'$ and $z_2\ast e'$
to $\mathcal{H}^\diamond_1$;
for each $e''\in \mathcal{W}_1^{\diamond''}$ add the two 2-simplices $z_1\ast e''$ and $z_2\ast e''$
to $\mathcal{H}^\diamond_1.$ 
To complete $\mathcal{H}^\diamond_1$ add the 2-simplices $\{z_3^jz_1z_0 ~|~ j=1, \ldots, 2n \}.$
}
\label{fig:Wdiamond1}
\end{center}
\end{figure}

\subsection{Fourth phase: filling the pillows, 
to obtain the PL-embedding of $\mathcal{H}^\star$ we seek:
$\mathcal{H}^\diamond_1, \mathcal{H}^\diamond_2,\ldots$  
$\mathcal{H}^\diamond_{n-1}, \mathcal{H}^\diamond_n= \mathcal{H}^\star$}

We start the fourth phase with $\mathcal{W}^\diamond_1$ and $\mathcal{H}^\diamond_1$, which is defined in
Fig. \ref{fig:Wdiamond1}.

\begin{proposition}
\label{cor:embedcone}
If $\mathcal{W}_1^\diamond$ is embedded rectilinearly in $\Pi'\cup \Pi''$,
then it can be 
extended to an embedding of $\mathcal{H}^\diamond_1$ into $\mathbb{R}^3$,
via the cone construction.
\end{proposition}
\begin{proof}
 Straightforward from the simple geometry of the situation. See Fig. \ref{fig:Wdiamond1}. 
\end{proof}

Let $\mathcal{L}_{i+1}^\star$ be a subset of the pillow $\mathcal{P}_{i+1}^\star$, formed by the part
that comes from the tail of the balloon after the i-th $bp$-move is applied, see Fig. \ref{fig:U3}.

\begin{figure}[H]
\begin{center}
\includegraphics[width=13cm]{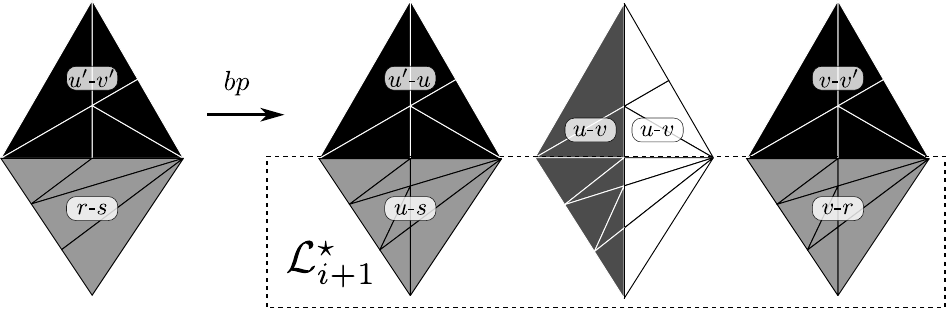}
\caption{\sf The set $\mathcal{L}_{i+1}^\star$'s: at each step of 
Theorem \ref{theo:teoremadeumalinha} we embed a set 
 $\mathcal{L}_{i+1}^\star$, $i=0,\ldots n-1$ of 2-simplices.
 These sets are the complementary parts of the PL2-faces 
 already in $\mathcal{H}^\diamond_1$ after a change of colors
 in the medial layer. 
The process of replacing the combinatorial tail of a balloon 
by the corresponding
trio of embedded PL2-faces in the pillow is denominated 
\index{blow up a tail} {\em the blowing up of the 
balloon's tail}.
 } 
\label{fig:U3}
\end{center}
\end{figure}

\begin{theorem}\label{theo:teoremadeumalinha}
 There is an $O(n)$-algorithm for blowing up a single balloon's tail.
Thus finding $\mathcal{H}_n^\star$ take, $O(n^2)$ steps.
\end{theorem}
\begin{proof}
$\mathcal{H}_{i+1}^\diamond$ is the union of $\mathcal{H}_{i}^\diamond$ with $\mathcal{L}_{i+1}^\star$
and an {\em $\epsilon$-change} in some PL3-faces, if the rank of the type of balloon's tail of the i-th $bp$-move has rank greater than 1
(we call $\epsilon$-change because this change is small, as described below).
At the same time we update the colors of the middle layer to match the colors of the $i$-th pillow in the sequence of $bp$-moves.


Now we describe how to embed each kind of $\mathcal{L}_{i}^\star$ 
(explaining how to $\epsilon$-change some PL3-faces,
to get space for $E\mathcal{L}_i^\star$).

If the balloon's tail is of type $P_1$ (the case $B_1$ is analogous).
Make two copies of $P_1$, resulting in three $P_1$, but change the color of the one which will be in
the middle, and define the 0-simplices like in Fig. \ref{fig:3d2}.

\begin{figure}[!htb] 
\begin{center}
\includegraphics[width=14cm]{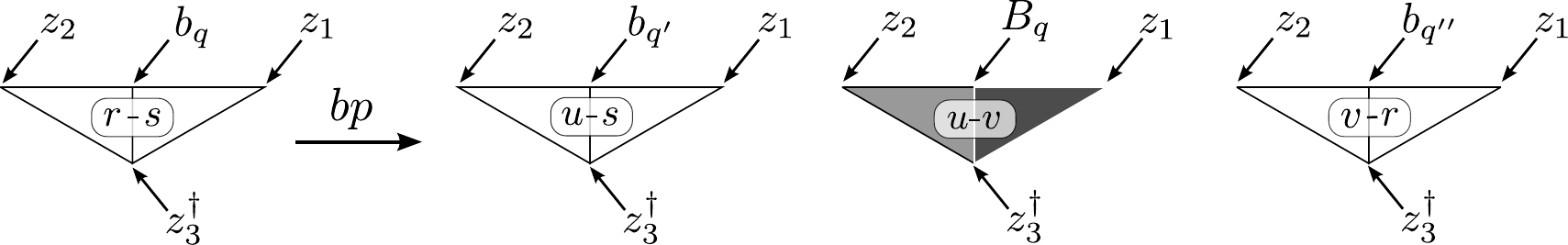}
\caption{\sf Embedding the part of the pillow corresponding to the tail of the balloon: case $P_1$ of the tail.}
\label{fig:3d2}
\end{center}
\end{figure}
If the balloon's tail is of type $B_i$, $i>1$ (the case $P_i$ is analogous).
Make two copies of $B_i$, refine the copies and the original, resulting in three $B_i'$, 
but change the color of the one which will be in
the middle, and define the 0-simplices like in Fig. \ref{fig:3d3}.

The images $\chi_j$ we already know from previous $bp$-move, now we need to define all the images
$\alpha_j, \beta_j$ and $\gamma_j.$
Let $\beta_j$ be $\frac{z_2+\chi_{j+1}}{2}$ 
for each $j=1,\ldots, i$.
As the images $\alpha_j$ and $\gamma_j$ can be defined in analogs way, we just explain how to define each $\alpha_j$.
We know that each $\alpha_j$ is in the PL3-face $\nabla_r$. To 
define each $\alpha_j$ we need 
to reduce the PL3-face $\nabla_r$
in order to get enough space for the PL2-faces of color 0 and 2 of the PL3-faces $\nabla_u$ and $\nabla_v$.
Consider the PL3-face $\nabla_r$, each $\beta_j$ is already defined, so 
define each $\zeta_j$ as $\frac{z_2+\omega_{j+1}}{2}$, where $\omega_k$ is previously defined, 
 see Fig. \ref{fig:nextdual3}. Define $\alpha_j$ as $\frac{\zeta_j+\beta_j}{2}$. 

\begin{figure}[!htb] 
\begin{center}
\includegraphics[width=14cm]{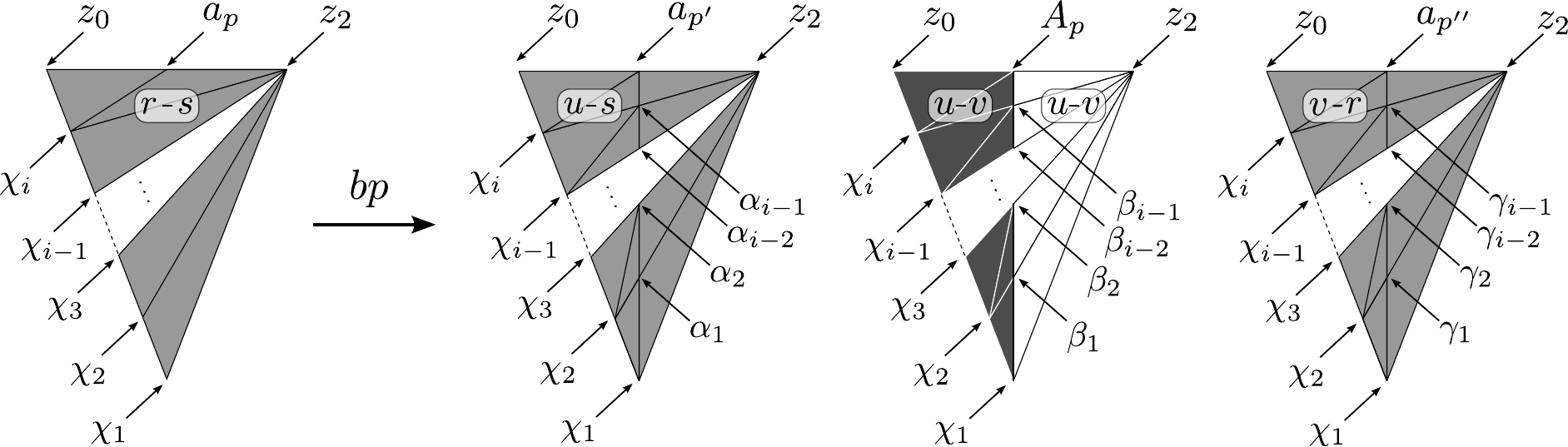}
\caption{\sf Embedding the part of the pillow corresponding to the tail of the balloon: case $B_i$ of the tail.}
\label{fig:3d3}
\end{center}
\end{figure}

\begin{figure}[!htb] 
\begin{center}
\includegraphics[scale=0.8]{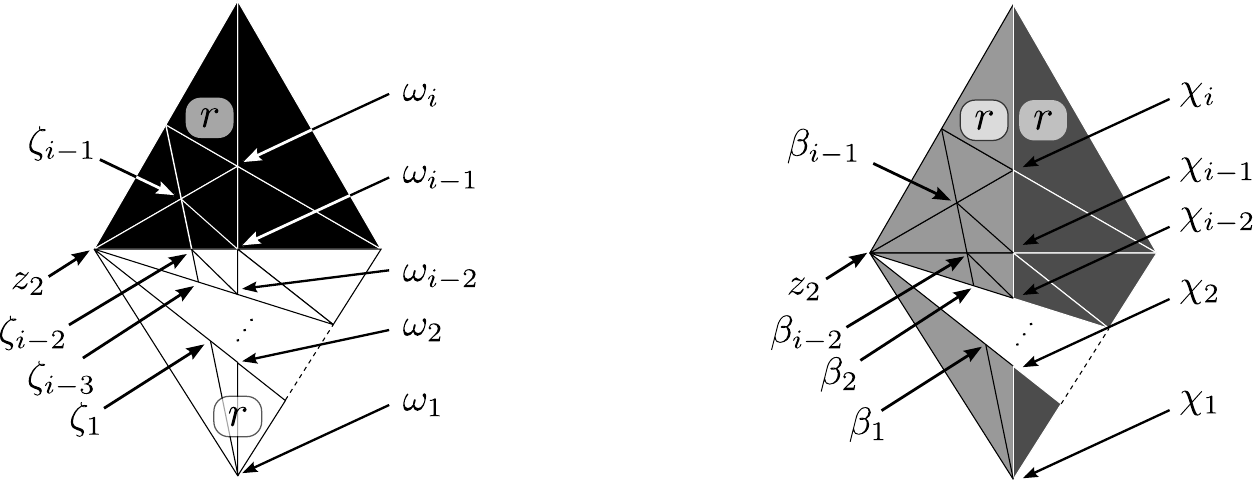}
\caption{\sf Using the PL3-face corresponding to $r$ to define the $\alpha_j$ as $\frac{\zeta_j+\beta_j}{2}$.}
\label{fig:nextdual3}
\end{center}
\end{figure}

The last case is when balloon's tail is refined, that means it is of type $P_i'$ or $B_i'$, $i>1$. We treat the case $B_i'$,
 see Fig. \ref{fig:3d4}. All the 0-simplices $\beta_j$ are already defined, we need to define each
$\alpha_j$ and each $\gamma_j$. Observe that here $r\neq s-1$ and the definitions of $\alpha_j$ and $\gamma_j$ 
are not analogous.

\begin{figure}[!htb] 
\begin{center}
\includegraphics[width=15cm]{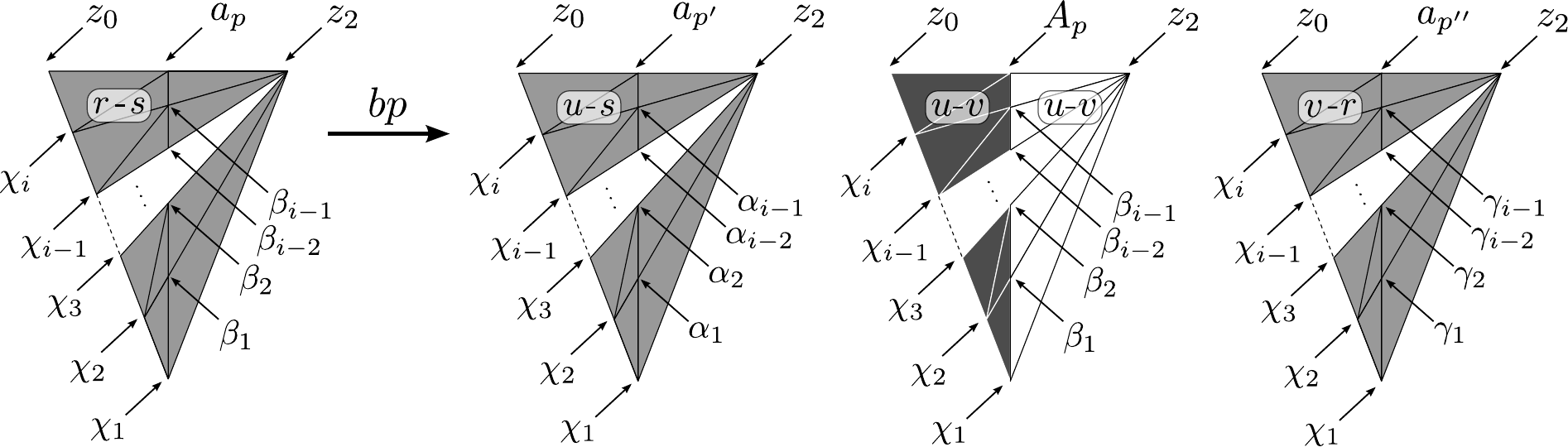}
\caption{\sf Embedding the part of the pillow corresponding to the tail of the balloon: case $B_i'$ of the tail.}
\label{fig:3d4}
\end{center}
\end{figure}

In this case, we need to reduce the PL3-faces $\nabla_r$ and $\nabla_s$
to create enough space to build PL2-faces 0- and 2-colored.
To define 0-simplices $\alpha_j$ and $\gamma_j$, one of these cases is analogous to the case not
 refined, but the other we describe here. ($\nabla_r$ is in the new case is the rank 
 of PL2$_0$-face is equals to the rank of the PL2$_1$-face plus 2, if its not true,
the new case is in the PL3-face $\nabla_v$).
Suppose that the new case is in the PL3-face, $\nabla_r$. To define $\alpha_j$,
suppose that the PL2$_0$-face of this PL3-face is not refined,
 see Fig. \ref{fig:3d5}. Define each $\alpha_j$  as the middle point between $\beta_j$ and $\omega_j$.

\begin{figure}[!htb] 
\begin{center}
\includegraphics[scale=0.7]{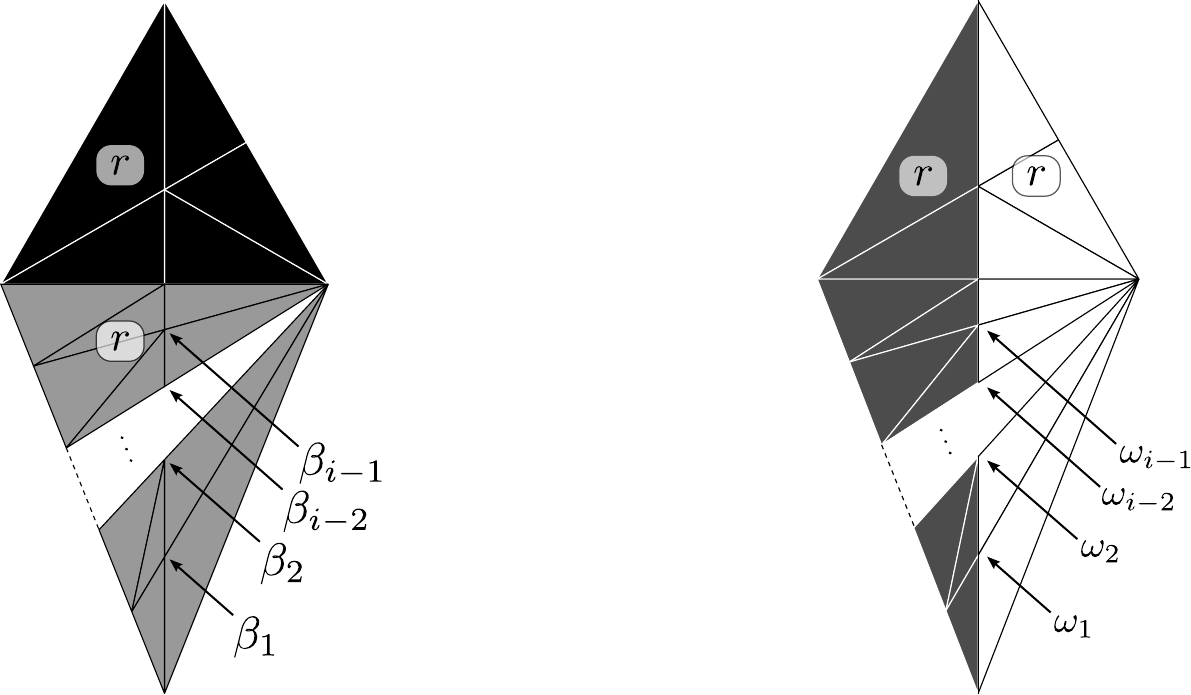}
\caption{\sf Using the PL3-face $\nabla_r$ to define $\alpha_j$ as $\frac{\omega_j+\beta_j}{2}$.}
\label{fig:3d5}
\end{center}
\end{figure}

Consider the case that the PL2$_0$-face, of the PL3-face $\nabla_r$, is refined
see Fig. \ref{fig:3d6}.
 This is a final subtlety which 
is treated with the {\em bump}. \index{bump} This is characterized by a
non-convex pentagon shown in the bottom part of Fig. \ref{fig:3d6}.
Let $\nu_j$ be $\frac{z_2+\omega_j}{2}$
and $\alpha_j$ as $\frac{\beta_{j-1}+\nu_j}{2}$, for $j=1,\ldots, i-1$. Observe that
if we define $\alpha_j$ as if the PL2$_0$-face
where not refined, some 1-simplices may cross.

\begin{figure}[!htb] 
\begin{center}
\includegraphics[scale=0.7]{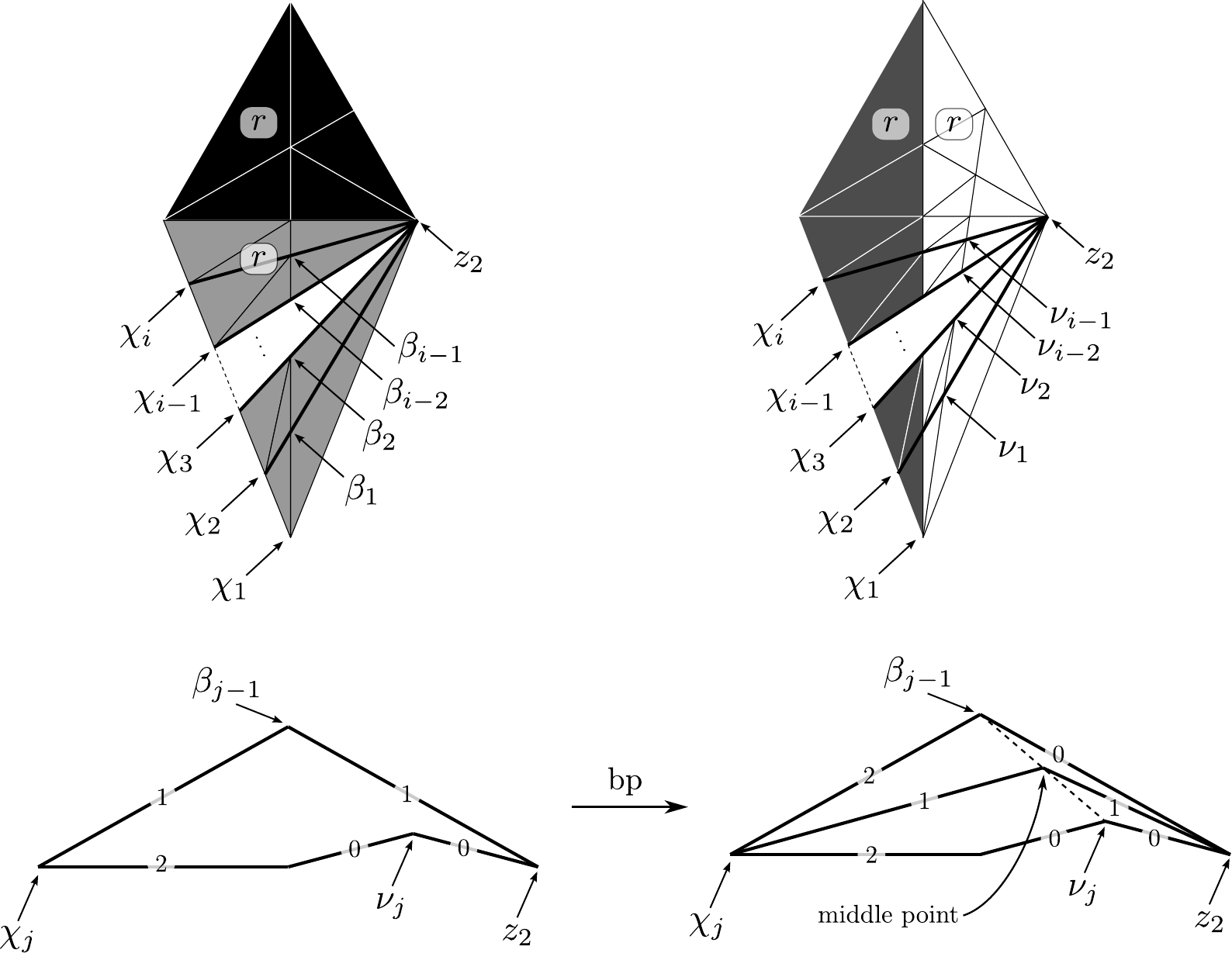}
\caption{\sf The bump: a final subtlety and how to deal with it.}
\label{fig:3d6}
\end{center}
\end{figure}

\end{proof}

\bibliographystyle{plain}
\bibliography{bibtexIndex.bib}

\vspace{10mm}
\begin{center}

\hspace{7mm}
\begin{tabular}{l}
   S\'ostenes L. Lins\\
   Centro de Inform\'atica, UFPE \\
   Av. Jornalista Aníbal Fernandes s/n \\
   Recife--PE 50740-560\\
   Brazil\\
   sostenes@cin.ufpe.br
\end{tabular}
\hspace{15mm}
\begin{tabular}{l}
   Ricardo N. Machado\\
   Núcleo de Formação de Docentes, UFPE\\
   Av. Jornalista Aníbal Fernandes s/n \\
   Caruaru--PE \\
   Brazil\\
   ricardonmachado@gmail.com
\end{tabular}

\end{center}

\end{document}